\RequirePackage[T1]{fontenc}
\documentclass[12pt,twoside]{amsart}
\calclayout
\usepackage[utf8]{inputenc}
\usepackage{amsmath}
\usepackage{amssymb}
\usepackage{mathtools}
\numberwithin{equation}{section}
\usepackage{slashed}
\usepackage{braket}
\usepackage[svgnames]{xcolor}
\usepackage{colortbl}
\definecolor{lightgray}{gray}{0.9}
\definecolor{xgray}{gray}{0.8}
\colorlet{conjcolor}{blue!5!white}
\usepackage[colorlinks,citecolor=DarkGreen,linkcolor=FireBrick,urlcolor=FireBrick,linktocpage,breaklinks=true]{hyperref}
\usepackage{cite}
\usepackage[pdftex]{graphicx}
\usepackage{tikz}
\usepackage{tikz-cd}
\usepackage{times}
\usepackage{courier}
\usepackage{bm}
\usepackage{subfig}
\usepackage{fnpct}

\usepackage{mathrsfs} 

\usepackage[all]{xy}
\usepackage{enumitem}

\theoremstyle{plain}
\newtheorem{defn}[equation]{Definition}

\newtheorem{thm}[equation]{Theorem}

\newtheorem{prop}[equation]{Proposition}

\def\PhysicsFact{Physics Assumption}

\newtheorem{fact?}[equation]{Fact?}
\newtheorem{cor}[equation]{Corollary}

\newtheorem{lem}[equation]{Lemma}

\newtheorem{conj}[equation]{Conjecture}

\theoremstyle{remark}
\newtheorem{rem}[equation]{Remark}

\newtheorem{const}[equation]{Construction}
\newtheorem{physremark}[equation]{Physics Remark}

\usepackage[many]{tcolorbox}

\tcbset{
  breakable,
  colback=conjcolor,
  boxrule=-1pt,
  boxsep=1pt,
  left=2pt,right=2pt,top=2pt,bottom=2pt,
  oversize=2pt,
  sharp corners,
  before skip=\topsep,
  after skip=\topsep,
}
\tcolorboxenvironment{proofc}{}
\tcolorboxenvironment{defnc}{}
\tcolorboxenvironment{propc}{}
\tcolorboxenvironment{thmc}{}
\tcolorboxenvironment{corc}{}

\let\oldendrem\endrem
\def\endrem{\hfill {$\lrcorner$} \oldendrem}
\let\oldendconst\endconst
\def\endconst{\hfill {$\lrcorner$} \oldendconst}
\let\oldendphysremark\endphysremark
\def\endphysremark{\hfill {$\lrcorner$} \oldendphysremark}

\let\oldtext\text
\def\text#1{\oldtext{\upshape\mdseries #1}}
\def\Nequals#1{$\mathcal{N}{=}#1$}
\def\bC{\mathbb{C}}
\def\bH{\mathbb{H}}
\def\bR{\mathbb{R}}
\def\bZ{\mathbb{Z}}
\def\cD{\mathcal{D}}
\def\cI{\mathcal{I}}
\def\cO{\protect{\mathcal{O}}}
\def\cT{\mathcal{T}}
\def\fg{\mathfrak{g}}
\def\sG{\mathsf{G}}
\def\Ker{\mathop{\mathrm{Ker}}}

\begin{document}

\title[On $\#\{\Gamma\to G\}$]{On homomorphisms
from finite subgroups of $SU(2)$ \\[.5em]
to Langlands dual pairs of groups}
\author{Yuki Kojima}
\author{Yuji Tachikawa}
\date{ May, 2025}
\address{Kavli Institute for the Physics and Mathematics of the Universe \textsc{(wpi)},
  The University of Tokyo,
  5-1-5 Kashiwanoha, Kashiwa, Chiba, 277-8583,
  Japan
}
\email{yuki.kojima@ipmu.jp, yuji.tachikawa@ipmu.jp}
\def\funding{
The research of YK is supported in part by 
 the Programs for Leading Graduate Schools, MEXT, Japan, via
the Leading Graduate Course for Frontiers of Mathematical Sciences and Physics.
The research of YT is  supported 
in part  by WPI Initiative, MEXT, Japan through IPMU, the University of Tokyo,
and in part by Grant-in-Aid for JSPS KAKENHI Grant-in-Aid (Kiban-C), No.24K06883.
}
\thanks{\emph{Acknowledgments}: 
The authors thank Sunjin Choi for the collaboration on the physics parts of this work.
\funding
}

\begin{abstract}
Let $N(\Gamma,G)$ be the number of homomorphisms from $\Gamma$ to $G$ 
up to conjugation by $G$.
Physics of four-dimensional \Nequals4 supersymmetric gauge theories
predicts that $N(\Gamma,G)=N(\Gamma , \tilde G)$
when $\Gamma$ is a finite subgroup of $SU(2)$,
$G$ is a connected compact simple Lie group
and $\tilde G$ is its Langlands dual. 
This statement is known to be true when $\Gamma=\bZ_n$,
but the statement for non-Abelian $\Gamma$ is new, to the knowledge of the authors.
To lend credence to this conjecture,
we  prove this equality 
in a couple of examples, namely
$(G,\tilde G)=(SU(n),PU(n))$ and $(Sp(n),SO(2n+1))$ for arbitrary $\Gamma$, 
and $(PSp(n),Spin(2n+1))$ for exceptional $\Gamma$.

A more refined version of the conjecture, together with proofs of some concrete cases,
will also be presented.
The authors would like to ask mathematicians to provide a more uniform proof
applicable to all cases.
\end{abstract}
\maketitle

\tableofcontents


\section{Introduction}

\subsection{A conjecture}
Let $G$ be a connected compact simple Lie group,
and $T$ be the Cartan torus of $G$.
Its lattice of characters, $M$, sits between the root lattice $Q$ and the weight lattice $P$ of $G$:
\[
Q\subset M \subset P \subset \mathfrak{t}^*,
\]
where $\mathfrak{t}$ is the Lie algebra of $T$.
Then, the lattice $M^*$ of cocharacters (i.e.~the kernel of the exponentiation map $\exp: \mathfrak{t}\to T$)
is naturally dual to $M$ and satisfies \[
P^* \subset M^* \subset Q^* \subset \mathfrak{t}.
\]
Here, for a  free finitely-generated Abelian group $A$, we denote by $A^*$ the dual free finitely-generated Abelian group such that there exists a perfect pairing $A\times A^*\to \bZ$.

Now, it is well-known that there is a connected compact simple Lie group $\tilde G$,
uniquely determined up to isomorphism,
such that its root lattice is $P^*$, the weight lattice is $Q^*$,
and the character lattice of its Cartan torus is $M^*$.
The groups $G$ and $\tilde G$ are said to be Langlands dual to each other.
Some examples include  \[
(G,\tilde G)=(SU(n), PU(n)),\quad
(Sp(n),SO(2n+1)),\quad
(PSp(n),Spin(2n+1)).
\label{1.1}
\]

The Langlands dual groups appear not only in number theory but also in physics of
four-dimensional supersymmetric quantum field theory,
and in this paper we will be interested in the following conjecture coming from the latter context.
To state the conjecture, we need some definitions.
\begin{defn}
Let $\Gamma$ be a finite group,
and $G$ be a connected compact simple Lie group.
Let $\phi(\Gamma,G)$ be the set of homomorphisms $f:\Gamma\to G$,
and we define two such homomorphisms $f_1$ and $f_2$ to be equivalent, $f_1\sim f_2$,
when they are conjugate by the action of $G$, i.e.~there is an element $g\in G$
such that \[
f_2(\gamma)  = g f_1(\gamma) g^{-1},\qquad \text{for all}\quad \gamma\in\Gamma.
\]
We then let \[
\psi(\Gamma,G) =   \phi(\Gamma,G)/\sim 
\] and \[
N(\Gamma,G)= \#\psi(\Gamma,G),
\]
i.e.~the number of homomorphisms $f:\Gamma\to G$ up to conjugation by $G$.
\end{defn}

\begin{conj}
\label{conj:rough}
Let $\Gamma$ be a finite subgroup of $SU(2)$,
and $(G,\tilde G)$ to be a Langlands dual pair of connected compact simple Lie groups.
Then \[
N(\Gamma, G)=N(\Gamma,\tilde G).
\]
\end{conj}

This is known to be true when $\Gamma=\bZ_n$, see \cite[Theorem 1]{Djokovic},
but this statement for non-Abelian $\Gamma$ is new, to the knowledge of the authors.
Our first aim of this paper is to provide evidence to this conjecture, by providing proofs 
of this conjecture in the following cases:
\begin{itemize}
\item $(G,\tilde G)=(SU(n),PU(n))$ for arbitrary $\Gamma$,
\item $(G,\tilde G)=(Sp(n),SO(2n+1))$ for arbitrary $\Gamma$,  and
\item $(G,\tilde G)=(PSp(n),Spin(2n+1))$ for exceptional $\Gamma$.
\end{itemize}

\begin{rem}
Although the conjecture as formulated above does not seem to be previously given in the literature,
the structure of $\psi(\Gamma,G)$ for some choices of $\Gamma$ and $G$ 
has been studied  in the past.
The paper \cite{Djokovic} given above is one.
The book \cite{Kac} has a very concrete method to describe each element of $\psi(\bZ_n,G)$ when
$G$ is of the adjoint type using the Dynkin diagram of $G$,  in its Sec.~8.6.
The studies \cite{FreyE6,Frey,FreyE7,Frey:2018vpw} have descriptions of $\psi(\Gamma,G)$ when $\Gamma$ is the binary icosahedral group and $G=E_6$, $E_7$ and $E_8$.
The last paper mentioned also has a description of $\psi(\Gamma,E_8)$
when $\Gamma$ is  some of the binary dihedral groups.
\end{rem}
\begin{rem}
As the partial proofs of this conjecture we give below will amply show, 
we do \emph{not} expect that there is a natural
bijection between $\psi(\Gamma,G)$ and $\psi(\Gamma,\tilde G)$.
It seems to be more like a discrete Fourier transformation on these two sets.
\end{rem}

\subsection{A more refined conjecture}

There is actually a finer version of the conjecture.
To state it, we need a few more preparations.  
Below, we use $A^\wedge$ to denote the Pontryagin dual of a finite Abelian group $A$. 
We denote the natural pairing $A^\wedge \times A\to U(1)=\{ z \mid |z|=1 \}$
by $(-,-)$.
Let $\Gamma$ be a finite group,
and $G$ be a group,
and $Z\subset G$ is an Abelian subgroup of the center of $G$.

Recall that elements of $H^2(B\Gamma;Z)$ classifies  central extensions \begin{equation}
\label{eq:ext}
0\to Z\to \check\Gamma\xrightarrow{p} \Gamma\to 0.
\end{equation} We pick and fix a particular $\check\Gamma$ for each class $w\in H^2(B\Gamma;Z)$.
Let us denote this extension by $\check\Gamma_{Z,w}$.
\begin{defn}
A \emph{$(Z,w)$-twisted homomorphism from $\Gamma$ to $G$}
is a homomorphism $f:\check\Gamma_{Z,w}\to G$ such that 
$Z\subset \check\Gamma_{Z,w}$ is mapped identically to $Z\subset G$.
We denote the set of $(Z,w)$-twisted homomorphisms from $\Gamma$ to $G$ by $\phi_{Z,w}(\Gamma,G)$.
We let $\psi_{Z,w}(\Gamma,G)$ be its quotient by $\sim$.
\end{defn}
\begin{rem}
\label{rem:triv}
It is clear that $\phi_{Z,0}(\Gamma,G)=\phi(\Gamma,G)$
and $\psi_{Z,0}(\Gamma,G)=\psi(\Gamma,G)$.
\end{rem}

Recall that elements of $H^1(B\Gamma;Z)$ are homomorphisms $z:\Gamma\to Z$.
Given a $w$-twisted homomorphism $f:\check\Gamma_{Z,w}\to G$,
we define $zf:\check\Gamma_{Z,w}\to G$ be the homomorphism defined by
\begin{equation}
zf(\bar\gamma) := z(\gamma)f(\bar\gamma),
\end{equation}  where $\gamma=p(\bar\gamma)$.
This defines an action of $H^1(B\Gamma,Z)$ on $\phi_{Z,w}(\Gamma,G)$ and $\psi_{Z,w}(\Gamma,G)$.

\begin{defn}
Let $V_{Z,w}(\Gamma,G)$  be a $\bC$-vector space with basis vectors $v([f])$ for $[f]\in\psi_{Z,w}(\Gamma,G)$,
and \[
V_Z(\Gamma,G):=\bigoplus_{w\in H^2(B\Gamma;Z)} V_{Z,w}(\Gamma,G).
\]
Let $z\in H^1(B\Gamma;Z)$ act on $V_Z(\Gamma,G)$ by \[
zv([f]) := v([zf]),
\]
and let $\hat w\in H^2(B\Gamma;Z)^\wedge$ act on $V_{Z,w}(\Gamma,G)$ by \[
\hat w v([f])= (\hat w,w) v([f]).
\]
This makes $V_Z(\Gamma,G)$ into a representation of \[
F(\Gamma;Z)  := H^1(B\Gamma;Z)\times H^2(B\Gamma;Z)^\wedge.
\]
\end{defn}

Let us now specialize to the case when $\Gamma$ is a finite subgroup of $SU(2)$.
$\Gamma$ has a natural action on $\bC^2$ coming from the embedding $\Gamma\subset SU(2)$.
It therefore acts on the unit sphere $S^3$ of $\bC^2$.
$S^3\to S^3/\Gamma$ is a principal $\Gamma$-bundle and determines a map $S^3/\Gamma\to B\Gamma$. 
It can be checked that, via this map $S^3/\Gamma\to BG$,
$H^1(B\Gamma;Z)$ and $H^2(B\Gamma;Z)$ 
pulls back isomorphically to $H^1(S^3/\Gamma;Z)$ and $H^2(S^3/\Gamma;Z)$, respectively.
Now, using Poincar\'e duality on $S^3/\Gamma$, 
one has $H^1(S^3/\Gamma;Z) \simeq H^2(S^3/\Gamma;Z^\wedge)^\wedge$.
Combining, we conclude that there is a naturally defined isomorphism
\begin{equation}
\iota_Z: H^1(B\Gamma;Z) \xrightarrow{\sim} H^2(B\Gamma;Z^\wedge)^\wedge.
\end{equation}
\begin{defn}
\label{defn:swap}
We let \[
F(\Gamma;Z) := H^1(B\Gamma;Z)\times H^2(B\Gamma;Z)^\wedge,
\]
and define the \emph{swap} isomorphism \[
s:  F(\Gamma;Z) \xrightarrow{\sim}  F(\Gamma;Z^\wedge)
\] by demanding that it sends an element \[
(z,\iota_{Z^\wedge}(z')) \in F(\Gamma;Z)=H^1(B\Gamma;Z)\times H^2(B\Gamma;Z)^\wedge
\] to \[
(z',\iota_{Z}(z)) \in F(\Gamma;Z^\wedge)=H^1(B\Gamma;Z^\wedge)\times H^2(B\Gamma;Z^\wedge)^\wedge.
\]
\end{defn}

We can now state a refinement of Conjecture~\ref{conj:rough}:
Let $G$ be a connected compact simple Lie group,
$Z\subset G$ be an Abelian subgroup of the center of $G$,
and $H=G/Z$,
so that we have a central extension of groups \[
0 \to Z \to G \to H \to 0.
\]
Then the Langlands duals $\tilde G$, $\tilde H$ of $G$, $H$ 
sit in a central extension 
\[
0\to Z^\wedge \to \tilde H \to \tilde G \to 0.
\]
Pick a finite subgroup $\Gamma$ of $SU(2)$.  Then
$V_Z(\Gamma,G)$ is a representation of $F(\Gamma;Z)$ 
and $V_{Z^\wedge}(\Gamma,\tilde H)$ is a representation of $F(\Gamma;Z^\wedge)$.
\begin{conj}
\label{conj:finer}
Under the setup above,
$V_Z(\Gamma,G)$ as a representation of $F(\Gamma;Z)$
and 
$V_{Z^\wedge}(\Gamma,\tilde H)$ as a representation of  $F(\Gamma;Z^\wedge)$
are equivalent when 
we identify the groups $F(\Gamma;Z)$ and $F(\Gamma;Z^\wedge)$
using the swap  isomorphism $s$.
\end{conj}

\begin{rem}
It is unclear to the authors if such an equivalence is canonically defined,
since it turns out that the action $\rho$ of $F(\Gamma;Z)$ on $V_Z(\Gamma,G)$
is such that $\rho$ and $\rho\circ \text{inv}$ are equivalent, 
where $\text{inv}$ is the automorphism
of $F(\Gamma;Z)$ given by sending $a\mapsto a^{-1}$.
\end{rem}

Again, to provide evidence to this conjecture, we  provide proofs of Conjecture~\ref{conj:finer}
in a couple of cases:
\begin{itemize}
\item arbitrary $(G,\tilde G)$, $(H,\tilde H)$, $Z$ for $\Gamma=\bZ_n$, 
\item arbitrary $\Gamma$ for \[
(G,\tilde G)=(SU(n),PU(n)), \quad
(H,\tilde H)=(PU(n),SU(n)),\quad
Z=\bZ_n,\]
\item and exceptional $\Gamma$ for \[
(G,\tilde G)=(Sp(n),SO(2n+1)),\quad
(H,\tilde H)=(PSp(n),Spin(2n+1)),\quad
Z=\bZ_2.\]
\end{itemize}
We note that Conjecture~\ref{conj:rough} follows from Conjecture~\ref{conj:finer} by setting $Z=\{e\}$
and taking the dimension, since $N(\Gamma,G)=\dim V_{\{e\}}(\Gamma,G)$.
There is also a less obvious relation between Conjecture~\ref{conj:rough} 
and Conjecture~\ref{conj:finer}, as can be seen from the proposition below.

\begin{prop}
\label{prop:gauging}
There is a natural identification \[
\begin{aligned}
V_Z(\Gamma,G)^{H^2(BG;Z)^\wedge} &\simeq V_{\{e\}}(\Gamma,G),\\
V_Z(\Gamma,G)^{H^1(BG;Z)} &\simeq V_{\{e\}}(\Gamma,G/Z),
\end{aligned}
\]
where  $V^G$ for a space $V$ with a $G$ action means
the subspace of $V$ invariant under $G$.
\end{prop}

\begin{proof}
To show the first isomorphism, note that $
V_Z(\Gamma,G)^{H^2(BG;Z)^\vee}
=V_{Z,0}(\Gamma,G)$. 
As $(Z,0)$-twisted homomorphisms are just genuine homomorphisms, 
we have $V_{Z,0}(\Gamma,G)=V_{\{e\}}(\Gamma,G)$.

To prove the second isomorphism, we first note that, by definition, 
any $(Z,w)$-twisted homomorphism $f$ from $\Gamma$ to $G$ descends to a genuine homomorphism from $\Gamma$ to $G/Z$.
Moreover, $f$ and $zf$ descends to the same homomorphism from $\Gamma$ to $G/Z$.

Conversely, given an arbitrary homomorphism $f: \Gamma\to G/Z$,
let us pick a point-wise lift of $f(\gamma)\in G/Z$ to $G$ for each $\gamma\in \Gamma$.
This defines a homomorphism from a certain extension $\check \Gamma$ of the form \eqref{eq:ext} to $G$.
This can be modified in a standard manner to a homomorphism $\check f$ 
from $\check\Gamma_w$ to $G$,
and the class $w\in H^2(B\Gamma;Z)$ is uniquely determined by $f$.
It is easy to check that two such $(Z,w)$-twisted homomorphisms $\check f$ and $\check f'$ from $\Gamma$ to $G$ 
come from a single $f:\Gamma\to G/Z$ if and only if there is a $z:\Gamma\to Z$ 
such that $\check f'=z\check f$.
\end{proof}

\begin{cor}
\label{cor:main}
Conjecture~\ref{conj:finer} for the data
$\Gamma$, $G$, $H{=}G/Z$, $\tilde H$, $\tilde G{=}\tilde H/Z^\wedge$
implies the equalities\[
N(\Gamma,G)=N(\Gamma,\tilde G),\qquad
N(\Gamma,H)=N(\Gamma,\tilde H),
\]i.e.~Conjecture~\ref{conj:finer} for the data $\Gamma$, $G$, $\tilde G$
and for the data $\Gamma$, $H$, $\tilde H$.
\end{cor}

\begin{proof}
Apply Proposition~\ref{prop:gauging} to Conjecture~\ref{conj:finer} and take the dimension.
\end{proof}

\subsection{Physics background}
Physics derivations of Conjectures \ref{conj:rough} and \ref{conj:finer}
will be detailed in \cite{ChoiTachikawaToAppear};
the analysis there was motivated and heavily influenced by earlier physics papers \cite{Ju:2023umb,Ju:2023ssy}.
Roughly, for any given $G$, there is a four-dimensional \Nequals4 supersymmetric 
quantum field theory $T(G,g)$, where $g$ is the coupling constant.
It is believed that $T(G,g)\simeq T(\tilde G,1/g)$,
which is called the S-duality of these theories.
The  most common physics approach  to quantum field theory 
uses Taylor expansion in the coupling constant $g$,
and therefore this equality cannot be seen within such a standard approach.
This S-duality is a non-Abelian generalization of the electromagnetic duality of the 
ordinary electromagnetism, i.e. the Abelian $U(1)$ gauge theory.

$V_{\{e\}}(\Gamma,G)$ is then the lowest energy subspace of the Hilbert space of the theory $T(G,g)$ on $S^3/\Gamma$.
As $T(G,g)\simeq T(\tilde G,1/g)$, we should have $V_{\{e\}}(\Gamma,G)\simeq V_{\{e\}}(\Gamma,\tilde G)$,
and taking the dimension, we obtain Conjecture~\ref{conj:rough}.
The finer conjecture, Conjecture~\ref{conj:finer},
is obtained by taking into account the action of 
what are called \emph{1-form symmetries} of the four-dimensional quantum field theory,
and the swap isomorphism is related to the electromagnetic duality.
Then Proposition~\ref{prop:gauging} describes the process of the gauging of such 1-form symmetries.

\subsection{Organization of the paper}
The rest of the paper is organized as follows.
\begin{itemize}
\item We begin in Sec.~\ref{sec:known} by recalling the proof of Conjecture~\ref{conj:rough} for $\Gamma=\bZ_n$ found in \cite{Djokovic}.
We also generalize the proof in \cite{Djokovic} to prove the refined version of the conjecture,
Conjecture~\ref{conj:finer}, for $\Gamma=\bZ_n$ in this section.
\item Then in Sec.~\ref{sec:A}, we prove Conjecture~\ref{conj:rough} for arbitrary $\Gamma$
and$(G,\tilde G)=(SU(n),PU(n))$.
We also prove Conjecture~\ref{conj:finer} for $G=\tilde H=SU(n)$, $H=\tilde G=PU(n)$ and arbitrary $\Gamma$.
\item Then in Sec.~\ref{sec:BC}, we prove Conjecture~\ref{conj:rough} for arbitrary $\Gamma$
for $(G,\tilde G)=(Sp(n),SO(2n+1))$.
\item In the final section, Sec.~\ref{sec:final}, we  prove Conjecture~\ref{conj:finer} for 
$(G,\tilde G)=(Sp(n),SO(2n+1))$, $(H,\tilde H)=(PSp(n),Spin(2n+1))$, and $\Gamma=\hat\cO $.
\end{itemize}

Some comments are in order.
\begin{itemize}
\item The readers will see that the proofs in Sec.~\ref{sec:known} and Sec.~\ref{sec:A} are somewhat conceptual
and follow a similar approach.
Namely, we rewrite $N(\Gamma,G)$ 
as the dimensions of $V^K$,
where $V$ is an auxiliary vector space, $K$ is an auxiliary finite group acting on it, and $V^K$ is the fixed point subspace.
Similarly, we rewrite $N(\Gamma,\tilde G)$ as the dimension of $\tilde V^{\tilde K}$,
for a suitable choice of $\tilde V$ and $\tilde K$.
Then we will show that there is a linear isomorphism $f:V\to\tilde V$
intertwining the actions of $K$ and $\tilde K$.
The map $f$ works as a discrete Fourier transform.
It is to be noted that the choice of $V$ and $K$ for a given $(\Gamma,G)$ in
 Sec.~\ref{sec:known} and Sec.~\ref{sec:A} would be different even for
 $(\Gamma,G)$ common to both sections.
\item In contrast, our proofs in Sec.~\ref{sec:BC} and Sec.~\ref{sec:final} are more combinatorial,
and utilize  computations using generating functions.
It would be nice if the proofs here could be rephrased in  a form closer to the proofs in Sec.~\ref{sec:known}
and Sec.~\ref{sec:A}.
\item In Sec.~\ref{sec:A} and later, 
when we deal with general $\Gamma$, we will rely heavily on McKay correspondence
between finite subgroups of $SU(2)$ and ADE Dynkin diagrams.
\end{itemize}

The authors are theoretical physicists.
They would hope that some mathematicians reading this paper would get interested and
eventually find a more conceptual and uniform proof of the conjectures.

\section{The case $\Gamma=\bZ_n$, arbitrary $(G,\tilde G)$}
\label{sec:known}
\subsection{The proof of the basic conjecture}
We first reproduce the proof of the following theorem, originally found in \cite{Djokovic},
see Theorem 1 there.
This is a subcase of our Conjecture~\ref{conj:rough} when $\Gamma=\bZ_n$.
\begin{thm}
Conjecture~\ref{conj:rough} holds when $\Gamma=\bZ_n$. 
That is, we have $N(\bZ_n,G)=N(\bZ_n,\tilde G)$ for arbitrary Langlands dual pair, $(G,\tilde G)$.
\end{thm}

\begin{proof}
Let $M^*$ and $M$ be the cocharacter lattice of $G$ and $\tilde G$, respectively.
Then clearly we have \[
\begin{aligned}
N(\bZ_n,G) &= \text{the number of $W$ orbits in}\ \tfrac1nM^*/ M^*,\\
N(\bZ_n,\tilde G) &= \text{the number of $W$ orbits in}\ \tfrac1nM/ M
\end{aligned}
\]
where $W$ is the Weyl group, common to both $G$ and $\tilde G$.
As $M$ and $M^*$ are dual to each other, 
$\tfrac1n M^*/M^*$ and $\tfrac 1n M/M\simeq M/nM$ are Pontryagin dual to each other.
In view of this fact, let us write $A=\tfrac1nM^*/ M^*$ and $A^\wedge=\tfrac 1nM/M$.
Let $V(A)$ be a complex vector space with basis $v(a)$ for each $a\in A$,
and similarly for $A^\wedge$.
We then have \[
	N(\bZ_n,G)=\dim V(A)^W, \qquad
	N(\bZ_n,\tilde G)=\dim V(A^\wedge)^W.
\]
As $V(A^\wedge)=V(A)^*$ and the actions of $W$ are conjugate, we have $N(\bZ_n,G)=N(\bZ_n,\tilde G)$.
\end{proof}

\subsection{The proof of the refined conjecture}

This proof can be readily generalized for the finer version of the conjecture:
\begin{thm}
\label{thm:first-finer}
Conjecture~\ref{conj:finer} holds when $\Gamma=\bZ_n$. 
\end{thm}
\begin{proof}
Let us recall the setup. 
We have  two Langlands dual pairs $(G,\tilde G)$ and $(H,\tilde H)$ sitting in the sequences \[
0\to Z\to G\to H\to 0,\qquad
0\to Z^\wedge\to \tilde H\to\tilde G\to 0.
\] 
Denoting the Cartan torus of $G$ by $T_G$, etc., we also have \[
0\to Z\to T_G\to T_H\to 0,\qquad
0\to Z^\wedge\to T_{\tilde H}\to T_{\tilde G}\to 0.
\]
Denote by $M$ and $N$ the character lattices of $G$ and $H$, respectively.
Then we have \[
Q\subset N\subset M \subset P\subset \mathfrak{t}^*,\qquad
P^* \subset M^*\subset  N^* \subset Q^*\subset \mathfrak{t},
\] and \[
N^*/M^* \simeq Z, \qquad M/N \simeq Z^\wedge.
\]

Let us first study $V_{Z}(\bZ_n,G)$.
In this case, $ H^2(B\bZ_n,Z)\simeq Z/nZ$,
and  therefore a $w\in H^2(B\bZ_n,Z)$ is specified by 
an element $\tilde w\in Z$ modulo $nZ$. We pick $\tilde w$ for each $w$.

A $(Z,w)$-twisted homomorphism from $\bZ_n$ to $G$ is
a $(Z,w)$-twisted homomorphism from $\bZ_n$ to $T_G$ up to the action of the Weyl group $W$.
It is specified by the image $a$ of $1\in \bZ_n$ such that $na = \tilde w \in Z$.
Therefore, using \[
(n \times) : \tfrac 1n N^*/M^* \to N^*/M^*=Z,
\] we have \[
\psi_{Z,w}(\bZ_n,T_G) \simeq (n\times)^{-1}(\tilde w).
\]
On this set, the group $H^1(B\bZ_n,Z)= \Ker Z\rightarrow{n\times }Z$ 
acts by the addition. 

Introduce now a vector space $V(\frac 1n N^*/M^*)$ 
with a basis $v(a)$ for each element $a\in \frac1n N^*/M^*$.
Identify the dual $(Z/nZ)^\wedge$ of $Z/nZ$  with $\Ker Z^\wedge \xrightarrow{n\times} Z^\wedge$.
Pick, then, an element $b \in Z^\wedge$ such that $nb=0$,
and let it act on $V(\frac1n N^*/M^*)$ via \[
b v(a) \mapsto (b, na) v(a).
\]
This makes $V(\frac 1n N^*/M^*)$ into a representation of 
$F(\bZ_n,Z)=H^1(B\bZ_n,Z)\times H^2(B\bZ_n,Z)^\wedge$,
and now it is routine to check that \[
V(\tfrac 1n N^*/M^*)^W = \text{$|nZ|$ copies of}\ V_Z(\bZ_n, G),
\] as a representation of $F(\bZ_n,Z)$.

We can similarly show that \[
V(\tfrac 1n M/N)^W = \text{$|nZ^\wedge|$ copies of}\ V_Z(\bZ_n, \tilde H) 
\] as a representation of $F(\bZ_n,Z^\wedge)$.
Using the Pontryagin duality between $\tfrac 1n N^*/M^*$ and $\tfrac 1n M/N$,
it is straightforward to check that 
the actions of $F(\bZ_n,Z)$ and $F(\bZ_n,Z^\wedge)$ on 
$V(\tfrac 1n N^*/M^*)$ and $V(\tfrac 1n M/N)$, respectively,
are compatible under the swap isomorphism of Definition~\ref{defn:swap}.
All what is left is to take the $W$-invariant parts.
\end{proof}

\section{The case $(G,\tilde G)=(SU(n),PU(n))$, arbitrary $\Gamma$}
\label{sec:A}

\subsection{Properties of irreducible representations of $\Gamma$}
In our discussion below, we heavily utilize the McKay correspondence
concerning irreducible representations of  
finite subgroups of $SU(2)$ and the ADE Dynkin diagram \cite{McKay}. 
Before proceeding to the rest of the paper,
we provide here a minimal amount of information.

Let $\Gamma$ be a finite subgroup of $SU(2)$.
Let $V$ be its defining 2-dimensional representation, i.e.~the one coming from the embedding $\Gamma\subset SU(2)$.
Let $\Phi=\{\rho_i\}$ be the set of isomorphism classes of irreducible representations  of $\Gamma$.
We now consider a graph, whose vertices are $\rho_i$'s,
such that $\rho_i$ and $\rho_j$ are connected if and only if $\rho_i\otimes V$ contains $\rho_j$ as an irreducible component.
This is known to produce an extended Dynkin diagram of type $A$, $D$ or $E$.
The finite subgroups of $SU(2)$ are then as follows:
 \begin{equation}
\begin{array}{c|c|c|c|c}
\text{ADE} & \text{symbol} & \text{order} &\text{name}&   \text{presentation} \\
\hline
 A_{n-1} & \bZ_n  & n& \text{cyclic} & a^n =1   \\
 D_{n+2} & \hat\cD_n & 4n&\text{binary dihedral} & a^2 =b^n=(ab)^2 \\
 E_6 & \hat\cT &24& \text{binary tetrahedral } & a^3 = b^3 = (ab)^2 \\
 E_7 & \hat\cO&48& \text{binary octahedral } & a^4 = b^3 = (ab)^2 \\
 E_8 & \hat\cI  &120& \text{binary icosahedral } & a^5 = b^3 = (ab)^2 \\
\end{array}\ .
\end{equation}
We will refer to them 
mainly by the symbols in the second column.
The presentations given above go back to \cite{CoxeterBinary}.

We note that $\dim \rho_i$ equals the comark of the extended Dynkin diagram.
We will utilize various other correspondences of the properties of $\Gamma$
and $\fg$.
They will be introduced as they become needed in our discussions.

\subsection{The proof of the basic conjecture}

The first objective of this section is to prove Conjecture~\ref{conj:rough} for arbitrary $\Gamma$
in the case $(G,\tilde G)=(SU(n),PU(n))$.
We start with some preparations.

Let  $\Gamma$ be a finite subgroup of $SU(2)$.
Let $A$ be the set of characters of $\Gamma_{ab}$, the Abelianization of $\Gamma$.
We have $A^\wedge = \Gamma_{ab}$.

Note $\psi(\Gamma,U(n))$ is the set of isomorphism classes of $n$-dimensional
complex representations of $\Gamma$.
Therefore, $A$ naturally acts on $\psi(\Gamma,U(n))$ by tensor product.
Let $V(\Gamma,U(n))$ be the complex vector space 
with basis $v\in[f]$ for each $[f]\in \psi(\Gamma,U(n))$.
$A$ then acts on $V(\psi(\Gamma,U(n)))$.

We now construct an action of $A^\wedge$ on $V(\Gamma,U(n))$.
\begin{defn}
\label{defn:det}
We let $\det:\psi(\Gamma,U(n))\to A$ be the map which sends a representation $\rho: \Gamma\to U(n)$ 
to its determinant, i.e.~the composition $\Gamma\xrightarrow{\rho} U(n)\xrightarrow{\det}U(1)$.
\end{defn}
Then we define $a^\wedge\in  A^\wedge$ to act on $v([f])$ via \begin{equation}
a^\wedge v([f]) = (a^\wedge, \det([f])) v([f]). \label{eq:action}
\end{equation}
\begin{rem}
Note that the $A$ action and the $A^\wedge$ action do not generally commute.
Rather, there is an action of a finite Heisenberg group which is an extension of $A\times A^\wedge$.
\end{rem}

\begin{lem}
\label{lem:1}
$N(\Gamma,SU(n))=\dim V(\Gamma,U(n))^{A^\wedge}$.
\end{lem}
\begin{proof}
Immediate.
\end{proof}

\begin{lem}
\label{lem:2}
$N(\Gamma,PU(n))=\dim V(\Gamma,U(n))^{A}$.
\end{lem}
\begin{proof}
Any homomorphism $f:\Gamma\to PU(n)$ 
can be lifted to a homomorphism $\tilde f: \Gamma\to U(n)$,
since $H^2(B\Gamma,U(1))=0$.
Suppose now two such lifts, $\tilde f$ and $\tilde f'$, descend to the same $f$.
Define $z(g)\in U(1)$ for $g\in \Gamma$ by the condition $\tilde f(g)=z(g) \tilde f'(g)$.
Then $z\in A$.
Therefore, $N(\Gamma,PU(n))$ is the number of $A$ orbits in $\psi(\Gamma,U(n))$.
The statement immediately follows from this.
\end{proof}

To compare $\dim V(\Gamma,U(n))^{A^\wedge}$ and $\dim V(\Gamma,U(n))^A$,
we use the McKay correspondence.
Let $\fg$ be the ADE type of $\Gamma$.
Then the irreducible representations $\rho_i$ of $\Gamma$
form the extended Dynkin diagram of type $\fg$.

An element of $\psi(\Gamma,U(n))$ specifies an $n$-dimensional representation $\rho$ of $\Gamma$ up to isomorphism.
As such it is specified by the number $n_i$ of the copies of the irreducible representation $\rho_i$ it contains.
These non-negative integers  $n_i$ satisfy $n_i\dim \rho_i =n$.
This is exactly what specifies an irreducible integrable highest-weight representation $\lambda$ of 
the affine Lie algebra $\hat\fg$ of level $n$.

\begin{defn}
\label{defn:McKay}
We denote by \[
\iota: \psi(\Gamma,U(n)) \xrightarrow{\sim} R(\hat\fg_n)
\] the McKay correspondence between two finite sets described above,
and call it the McKay map.
\end{defn}

\begin{rem}
It might be worth mentioning here that 
a deeper connection between $\hat\fg_n$ and $U(n)$ bundles on $\bC^2/\Gamma$
is known to exist, see e.g.~\cite{Nakajima}.
\end{rem}

Note that $A$ is the subset of $\Phi$ such that the comarks are $1$.
It is known that $A$ is naturally the outer automorphism group of the affine Lie algebra $\hat\fg$,
and therefore acts on $R(\hat\fg_n)$.

\begin{lem}
\label{lem:3}
The $A$ action on $\psi(\Gamma,U(n))$ and the $A$ action on $R(\hat\fg_n)$ are compatible 
under the McKay map $\iota$.
\end{lem}
\begin{proof}
Well known and can be checked by a case-by-case inspection.
\end{proof}

It is also known that $A^\wedge=\Gamma_{ab}$ is naturally isomorphic to the center $\mathsf{Z}$ 
of the simply-connected group $\sG$ of type $\fg$.
Using this, we make the following definition:
\begin{defn}
\label{defn:det'}
The irreducible decomposition of an irreducible integrable highest weight representation of $\hat \fg_n$
contains only a single type of irreducible representation of $\mathsf{Z}$.
We let $\det'$ be the map which associates this element in $\mathsf{Z}^\wedge=A$
to an element in $R(\hat\fg_n)$.
\end{defn}

\begin{lem}
\label{lem:4}
The map $\det:\psi(\Gamma,U(n)) \to A$  of Definition~\ref{defn:det}
and the map $\det': R(\hat\fg_n)\to A$ of Definition~\ref{defn:det'}
are compatible under the McKay map $\iota$.
\end{lem}
\begin{proof}
Well known and can be checked by a case-by-case inspection.
\end{proof}

We now introduce a vector space $V(\hat\fg_n)$ with basis $v(\lambda)$ 
for each $\lambda\in R(\fg_n)$.
$A$ and $A^\wedge$ naturally act on $V(\hat\fg_n)$.

\begin{lem}
\label{lem:combined}
The $A$ actions on  $V(\Gamma,U(n))$  and on $V(\fg_n)$ are compatible under the McKay map $\iota$.
Similarly, 
the $\hat A$ actions on  $V(\Gamma,U(n))$  and on $V(\fg_n)$ are compatible under the McKay map $\iota$.
\end{lem}

\begin{proof}
Immediate from Lemma~\ref{lem:3} and Lemma~\ref{lem:4}.
\end{proof}

\begin{defn}
We define the modular $S$-matrix \[
S: V(\hat\fg_n)\to V(\hat\fg_n)
\]
by $Sv(\lambda)=\sum_\mu S_{\lambda \mu} v(\mu)$, so that \[
S \chi_\lambda(-1/\tau)=\sum_\mu S_{\lambda \mu} \chi_\mu(\tau).
\]
Here $\chi_\lambda(\tau)$ 
is the character of the irreducible integrable highest-weight representation $\lambda$.
\end{defn}

\begin{prop}
The $A$ action on $V(\hat\fg_n)$ 
and the $A^\wedge$ action on $V(\hat\fg_n)$ is conjugate by the action of the modular $S$ matrix,
for a suitable identification $A\simeq A^\wedge$.
\end{prop}

\begin{proof}
See \cite[Sec.~14.6.4]{DiFrancesco:1997nk}.
\end{proof}

\begin{rem}
It seems to the authors that there is no canonical isomorphism $A\simeq A^\wedge$.
Two isomorphisms $A\xrightarrow{\sim}A^\wedge$ differing by
composing with $a\mapsto -a$ either on the side of $A$  or on the side of $A^\wedge$ seems
to give an equally good isomorphisms.
\end{rem}

\begin{cor}
\label{cor:conjugate}
The $A$ action and the $A^\wedge$ action on $V(\Gamma,U(n))$ are conjugate.
\end{cor}

\begin{proof}
Immediate from the proposition above and Lemma~\ref{lem:combined}.
\end{proof}

\begin{thm}
\label{thm:mainA}
We have $N(\Gamma,SU(n))=N(\Gamma,PU(n))$ for arbitrary finite subgroup $\Gamma$ of $SU(2)$.
\end{thm}

\begin{proof}
This follows from Lemma~\ref{lem:1}, Lemma~\ref{lem:2} and Corollary~\ref{cor:conjugate}.
\end{proof}

\subsection{The proof of the refined conjecture}

It is also not difficult to generalize the proof above to show the finer version of the conjecture 
when $(G,\tilde G)=(SU(n),PU(n))$,  $Z=\bZ_n$, and $(H,\tilde H)=(PU(n),SU(n))$.
As a preparation, we first need to study $H^1(B\Gamma,Z)$ and $H^2(B\Gamma,Z)$
when $Z=\bZ_n$, the center of $SU(n)$.
Recall that we defined $A=H^1(B\Gamma,U(1))$, i.e. the set of homomorphisms $\Gamma\to U(1)$.
Then we can identify $H^1(B\Gamma,\bZ_n) = \Ker (n\times): A\to A$.

Next, we consider $H^2(B\Gamma,\bZ_n)$. Take a cocycle representative $\omega(g,h)\in \bZ_n$.
As $H^2(B\Gamma,U(1))=0$, we can take $\nu(g)\in U(1)$ such that \[
\omega(g,h)=\nu(g)\nu(h)/\nu(gh).
\] To such $\nu$'s are different by an element of $A$.
Note also that $\nu(g)^n$ is a homomorphism. Let us denote it by $a\in A$.
This is well-defined up to $nA$. 
In this manner, we have defined a map $H^2(B\Gamma,\bZ_n)\to A/nA$.
This is actually an isomorphism, as can be seen by studying the Bockstein long exact sequence.

We consider a $(Z,\omega)$-twisted homomorphism from $\Gamma$ to $SU(n)$
as a projective representation $\rho$ of $\Gamma$ in $SU(n)$ such that \[
\rho(gh)=\omega(g,h) \rho(g) \rho(h). 
\]
Now define $\tilde \rho(g):= \nu(g) \rho(g)\in U(n)$. 
Then $\tilde \rho : \Gamma\to U(n)$ is a genuine representation.
Furthermore, \[
\det\tilde\rho(g)= \nu(g)^n =a(g),
\] meaning that $\det\tilde\rho$ is a one-dimensional representation of $\Gamma$.

Conversely, suppose we are given a genuine representation $\tilde \rho:\Gamma\to U(n)$ such that $\det\tilde\rho = a \in A$.
Let $\rho(g):=\nu(g)^{-1}\tilde\rho(g)$.
This defines a $(Z,w)$-twisted  homomorphism from $\Gamma$ to $SU(n)$.

In this manner we established the following proposition:
\begin{prop}
$\psi_{Z,w}(\Gamma,SU(n))$ can be identified with the subset $\det^{-1}(a)$ of $\psi(\Gamma,U(n))$,
where $a\in A$ is a representative of $w$ under the map $A\to A/nA \simeq H^2(B\Gamma,\bZ_n)$.
\end{prop}
The action of $H^1(B\Gamma,\bZ_n)$ on $\psi_{Z,w}(\Gamma,SU(n))$
is the action of $\Ker (n\times) : A\to A$ by the tensor product,

Using the fact that $(A/nA)^\wedge = \Ker (n\times): A^\wedge\to A^\wedge$, we obtain the 
following proposition:
\begin{prop}
There is a natural identification of two vector spaces,
\[
V(\Gamma,U(n))\quad \text{and}\quad \text{$|nA|$ copies of}\ V_Z(\Gamma,SU(n)),
\]
as  representations of $H^1(B\Gamma,\bZ_n)\times H^2(B\Gamma,\bZ_n)^\wedge$,
where \[
H^1(B\Gamma,\bZ_n)\simeq \Ker (n\times): A\to A
\] acts via the $A$ action on $V(\Gamma,U(n))$
and \[
H^2(B\Gamma,\bZ_n)^\wedge \simeq \Ker (n\times): A^\wedge\to A^\wedge
\]
 acts via the $A^\wedge$ action on $V(\Gamma,U(n))$.
\end{prop}

\begin{thm}
Conjecture~\ref{conj:finer} holds when $(G,\tilde G)=(SU(n),PU(n))$,
$(H,\tilde H)=(PU(n),SU(n))$ and $Z=\bZ_n$.
\end{thm}
\begin{proof}
Immediate from the proposition above and Corollary~\ref{cor:conjugate}.
\end{proof}

\section{The case $(G,\tilde G)=(Sp(n),SO(2n+1))$, arbitrary $\Gamma$}
\label{sec:BC}
In this section we prove:
\begin{thm}
\label{thm:mainA}
 Conjecture~\ref{conj:rough} holds for arbitrary $\Gamma$
when  $(G,\tilde G)=(Sp(n),SO(2n+1))$.
Namely, we have $N(\Gamma,Sp(n))=N(\Gamma,SO(2n+1))$ for arbitrary finite subgroup $\Gamma$ of $SU(2)$.
\end{thm}

We need some preparations.

\subsection{Generalities}

Given a finite group $\Gamma$, let us denote by $\rho_i$ its irreducible representations over $\bC$.
\begin{itemize}
\item When $\overline{\rho_i} \not\simeq \rho_i$, we call $\rho_i$ a complex representation.
We denote $\overline{\rho_i}$ by $\rho_{\bar i}$, and call $(\rho_i,\rho_{\bar i})$ a complex pair.
\item When $\overline{\rho_i} \simeq \rho_i$,
we call $\rho_i$ a real representation (in a wider sense). 
Furthermore, \begin{itemize}
\item if  $\rho_i$ is a complexification of a representation over $\bR$,
we call $\rho_i$ a strictly real representation.
\item if not, we call $\rho_i$ a pseudoreal representation. In this case, $\rho_i$ is obtained by 
taking a quaternionic action of $\Gamma$ on a quaternionic vector space $\bH^n$
and regarding it as an action on $\bC^{2n}$.
\end{itemize}
\end{itemize}

Consider a homomorphism $\rho: \Gamma \to Sp(n)$.
We have an action of $\Gamma$ on $\bH^n$. Regarding $\bH=\bC^2$,
we have an action of $\Gamma$ on $\bC^{2n}$, and then we can decompose it into irreducibles.
Let us say $\rho$ contains $n_i$ copies of $\rho_i$, so that $2n=\sum_i n_i\dim \rho_i$.
This lifts to an action on $\bH^{n}$ if and only if \begin{itemize}
\item $n_i$ is even for all strictly real $\rho_i$, and
\item $n_i=n_{\bar i}$ for all complex conjugate pairs $\rho_i$ and $\rho_{\bar i}$.
\end{itemize}

Similarly, consider  a homomorphism $\rho: \Gamma \to SO(2n+1)$.
This gives an action  of $\Gamma$ on $\bR^{2n+1}$. After complexification, 
we have an action of $\Gamma$ on $\bC^{2n+1}$, and then we can decompose it into irreducibles.
Let us say $\rho$ contains $n_i$ copies of $\rho_i$, so that $2n+1=\sum_i n_i\dim \rho_i$.
This lifts to an action on $\bR^{2n+1}$ if and only if \begin{itemize}
\item $n_i$ is even for all pseudoreal $\rho_i$, and
\item $n_i=n_{\bar i}$ for all complex conjugate pairs $\rho_i$ and $\rho_{\bar i}$.
\end{itemize}
This only guarantees that it is a homomorphism $\rho:\Gamma\to O(2n+1)$.
To guarantee that it is a homomorphism into $SO(2n+1)$, we need to require that $\det\rho$ is a trivial representation.

\begin{rem}
Note that the roles of 
strictly real irreducible representations ($\mathbb{R}$)
and  pseudoreal irreducible  representations ($\mathbb{H}$) are swapped 
when constructing general real representations ($\mathbb{R}$) 
and general  pseudoreal representations ($\mathbb{H}$).
This observations will have many repercussions in the generating functions we see below.
\end{rem}

\subsection{More properties of irreducible representations of $\Gamma$}

We now need to know the reality properties of complex irreducible representations of $\Gamma$.
The trivial representation is strictly real,
while the defining 2-dimensional representation $V$ coming from $\Gamma\subset SU(2)$ is clearly pseudoreal,
since $SU(2)\simeq Sp(1)$ naturally acts on $\bH$.
The irreducible decomposition after tensoring by $V$ then allows us to determine
the reality properties of many of the irreducible representations in a straightforward manner.

\subsubsection{$\Gamma=\bZ_n$.}
For $\bZ_{2n+1}$, there are $2n+1$ irreducible representations 
\[
\rho_k(a)=e^{2\pi i k/(2n+1)}, \qquad k=-n,-n+1,\ldots, +n.
\]
$\rho_0$ is a strictly real representation, and $\rho_{\pm k}$ for $k\neq 0$ are complex pairs of representations.

For $\bZ_{2n}$,
there are $2n$ irreducible representations $\rho_k$ for $k=-n+1,\ldots, n-1,n$, given by \[
\rho_k(a)=e^{2\pi i k/(2n)}.
\]
$\rho_0$ and $\rho_{n}$ are strictly real.
The rest are complex pairs, such that $\overline{\rho_k}=\rho_{-k}$.

\subsubsection{$\hat \cD_m$}

Let us first describe the structure of $\hat \cD_m$ common to both even $m$ and odd $m$.
There are $m-1$ two-dimensional representations $\rho_k$ for $k=1,\ldots, m-1$, given by \[
\rho_{2,k}(a)=\begin{pmatrix}0&i^k\\i^k&0\end{pmatrix},\qquad
\rho_{2,k}(b)=\begin{pmatrix}\alpha^k&0\\0&\alpha^{-k}\end{pmatrix}
\]
where $\alpha=e^{\pi i/m}$.
We can similarly define $\rho_{2,0}$ and $\rho_{2,m}$, but they decompose into two one-dimensional representations:
\[
	\rho_{2,0}=\rho_1\oplus \rho_{1'},\quad
	\rho_{2,m}=\rho_{1''}\oplus \rho_{1'''}.
\]
Explicitly, we have \[
	\begin{array}{c|cccc}
		& 1 & 1' & 1'' & 1'''\\
		\hline
		a & 1 & -1 & i^m & -i^m\\
		b & 1 & 1 & -1 & -1
	\end{array}.
\]
These are the irreducible representations of $\hat\cD_m$.
They form the extended Dynkin diagram of type $D_{n+2}$: \[
\begin{array}{c@{}c@{}c@{}c@{}c}
&1'& &1''& \\
&|& &|& \\
1-&2_1&-2_2-2_3-\cdots-&2_{m-1}&-1'''
\end{array}.
\]

Let us now specialize to $\hat\cD_{2n+1}$.
In this case, \begin{itemize}
\item $\rho_1$ and $\rho_{1'}$ are strictly real. 
\item $\rho_{1''}$ and $\rho_{1'''}$ form a complex pair.
\item $\rho_{2,k}$ for odd $k$ is pseudoreal, 
while $\rho_{2,k}$ for even $k$ is strictly real.
From explicit computations, $\det\rho_{2,k}$ for odd $k$ is $1$ and for even $k$ is $1'$.
\end{itemize}

Let us next discuss  $\hat\cD_{2n}$.
In this case, \begin{itemize}
\item $\rho_1$, $\rho_{1'}$, $\rho_{1''}$ and $\rho_{1'''}$ are all strictly real. 
\item $\rho_{2,k}$ for odd $k$ is pseudoreal, 
while $\rho_{2,k}$ for even $k$ is strictly real.
\item From explicit computations, $\det\rho_{2,k}$ for odd $k$ is $1$ and for even $k$ is $1'$.
\end{itemize}

\subsubsection{$\Gamma=\hat\cT $}
For $\Gamma=\hat\cT $ of type $E_6$,
the irreducible representations can be displayed according to the McKay correspondence as
\[
\begin{array}{c@{}c@{}c}
&1''&\\
&|&\\
&2''&\\
&|&\\
1-2-& 3 & -2'-1'
\end{array}
\]
where we used the dimension and additional primes if necessary to label them.
$1$ is the trivial representation and $2$ is the defining representation 
coming from the embedding $\Gamma\subset SU(2)$.
$1$ and $3$ are strictly real representations,
$2$ is a pseudoreal representation,
while  ($1'$, $1''$) and $(2',2'')$ are complex pairs of irreducible representations.
$\det$ maps $2$, $3$ to $1$, $2'$ to $1''$ and $2''$ to $1'$.

\subsubsection{$\Gamma=\hat\cO $}
For $\Gamma=\hat\cO $ of type $E_7$,
the irreducible representations can be displayed according to the McKay correspondence as
\[
\begin{array}{r@{}c@{}l}
	&2''&\\
	&|&\\
	1-2-3-&4&-3'-2'-1'
\end{array}.
\]
The $1$ is the trivial representation,
the $2$ is the 2-dimensional representation coming from the defining inclusion $\Gamma_{E_7}\subset SU(2)$.
The $3$ is the 3-dimensional representation obtained 
by applying the projection $SU(2) \to SO(3)$ to $2$.
The $1'$ is the sign representation $\rho_{1'}(a)=-1$ and $\rho_{1'}(b)=1$.
Then $2'=2\otimes 1'$ and $3'=2\otimes 1'$.
Finally, the $2''$ is given by \begin{equation}
\rho_{2''}(a)=\begin{pmatrix}1&0\\0&-1\end{pmatrix},\quad 
\rho_{2''}(b)=\text{$120^\circ$ rotation}.
\label{explicit-rep}
\end{equation}
From these descriptions, we know that $1$, $1'$, $3$, $3'$ and $2''$ are 
strictly real representations, while 
$2$, $2'$ and $4$ are pseudoreal representations.
$\det$ of pseudoreal representations are all $1$.

Among strictly real representations, $\det$ maps $1$ and $3$ to $1$,
and the rest to $1'$.
Checking this needs some work. 
We use the explicit presentation of $\hat\cO $.
As $b^3=a^4$, $\det(\rho(b))^3=\det(\rho(a))^4$.
Noting $\det\rho(a)=\pm1$, this forces $\det(\rho(b)^3)=1$, so $\det(\rho(b))=1$.
As for $\det\rho(a)$, we know explicitly that 
$\det\rho(a)=1$ for $1$ (as it is a trivial representation)
and for $3$ (as it is the projection from the defining representation in $SU(2)$ to $SO(3)$). 
We know $\rho(a)=-1$ for $1'$.
Therefore $\det\rho(a)=-1$ for $3'$.
For $2''$, we use the explicit representative given above 
to find $\det\rho(a)=-1$.

\subsubsection{$\Gamma=\hat\cI $}
Finally, for $\Gamma=\hat\cI $ of type $E_8$,
the irreducible representations can be displayed as 
\[
\begin{array}{c@{}c@{}c}
&3'&\\
&|&\\
1-2-3-4-5-&6&-4'-2'
\end{array}.
\]
$1$, $3$, $5$, $4'$, and $3'$ are strictly real representations,
and $2$, $4$, $6$, $2'$ are pseudoreal  representations.

\subsection{The proof of the basic conjecture}

From these data and the discussions above, we have the following result for the generating functions:
\begin{prop}
\label{prop:genX}
We have the following generating functions
for $N(\Gamma,G)$, for $G=Sp(n)$ and $G=SO(2n+1)$:
\begin{itemize}
\item For $\Gamma=\bZ_m$,
 \[
\begin{aligned}
\sum_{n=0}^\infty q^{2n} N(\mathbb{Z}_{m},Sp(n)) &= \frac1{(1-q^2)^{\lfloor\frac{m}{2}\rfloor+1}}, \\
\sum_{n=0}^\infty q^{2n+1} N(\mathbb{Z}_{m},SO(2n+1)) &= \frac1{2}\sum_{a=0}^1 (-1)^a\frac1{1-(-1)^{a}q}\frac1{(1-(-1)^{2a}q^2)^{\lfloor\frac{m}{2}\rfloor}}.
\end{aligned}
\]
\item For $\Gamma=\hat D_{2k}$,
\[
\begin{aligned}
\sum_{n=0}^\infty q^{2n} N(\hat D_{2k},Sp(n)) &= \frac1{(1-q^2)^{k+4}}\frac1{(1-q^4)^{k-1}},  \\
\sum_{n=0}^\infty q^{2n+1} N(\hat D_{2k},SO(2n+1)) &= \frac1{2}\sum_{a=0}^1\frac1{2^2}\sum_{b_0,b_1=0}^1 (-1)^a\frac1{1-(-1)^{a}q}\frac1{1-(-1)^{a+b_0}q}\frac1{1-(-1)^{a+b_1}q} \\
&\hspace{35pt} \times\frac1{1-(-1)^{a+b_0+b_1}q}\frac1{(1-(-1)^{2a+b_0}q^2)^{k-1}}\frac1{(1-(-1)^{4a}q^4)^k}.
\end{aligned}
\]
\item For $\Gamma=\hat D_{2k+1}$,
\[
\begin{aligned}
\sum_{n=0}^\infty q^{2n} N(\hat D_{2k+1},Sp(n)) &= \frac1{(1-q^2)^{k+3}}\frac1{(1-q^4)^k}, \\
\sum_{n=0}^\infty q^{2n+1} N(\hat D_{2k+1},SO(2n+1)) &= \frac1{2}\sum_{a=0}^1\frac1{2}\sum_{b=0}^1(-1)^a\frac1{1-(-1)^{a}q}\frac1{1-(-1)^{a+b}q}\frac1{1-(-1)^{2a}q^2} \\
&\hspace{120pt} \times\frac1{(1-(-1)^{2a+b}q^2)^{k}}\frac1{(1-(-1)^{4a}q^4)^k}.
\end{aligned}
\]
\item For $\Gamma=\hat\cT $,
\[
\begin{aligned}
\sum_{n=0}^\infty q^{2n} N(\hat\cT ,Sp(n)) &= \frac1{(1-q^2)^3}\frac1{1-q^4}\frac1{1-q^6}, \\
\sum_{n=0}^\infty q^{2n+1} N(\hat\cT ,SO(2n+1)) &= \frac1{2}\sum_{a=0}^1 (-1)^a\frac1{1-(-1)^{a}q}\frac1{1-(-1)^{2a}q^2}\frac1{1-(-1)^{3a}q^3} \\
&\hspace{210pt}\times\frac1{(1-(-1)^{4a}q^4)^2}, \\
\end{aligned}
\]
\item For $\Gamma=\hat\cO $,
\[
\begin{aligned}
\sum_{n=0}^\infty q^{2n} N(\hat\cO ,Sp(n)) &= \frac1{(1-q^2)^4}\frac1{(1-q^4)^2}\frac1{(1-q^6)^2}, \\
\sum_{n=0}^\infty q^{2n+1} N(\hat\cO ,SO(2n+1)) &= \frac1{2}\sum_{a=0}^1\frac1{2}\sum_{b=0}^1(-1)^a\frac1{1-(-1)^{a}q}\frac1{1-(-1)^{a+b}q}\frac1{1-(-1)^{2a+b}q^2} \\
&\hspace{5pt} \times\frac1{1-(-1)^{3a}q^3}\frac1{1-(-1)^{3a+b}q^3}\frac1{(1-(-1)^{4a}q^4)^2}\frac1{1-(-1)^{8a}q^8}, \\
\end{aligned}
\]
\item For $\Gamma=\hat\cI $,
\[
\begin{aligned}
\sum_{n=0}^\infty q^{2n} N(\hat\cI ,Sp(n)) &= \frac1{(1-q^2)^3}\frac1{1-q^4}\frac1{(1-q^6)^3}\frac1{1-q^8}\frac1{1-q^{10}}, \\
\sum_{n=0}^\infty q^{2n+1} N(\hat\cI ,SO(2n+1)) &= \frac1{2}\sum_{a=0}^1 (-1)^a\frac1{1-(-1)^{a}q}\frac1{(1-(-1)^{3a}q^3)^2}\frac1{(1-(-1)^{4a}q^4)^3} \\
&\hspace{80pt} \times\frac1{1-(-1)^{5a}q^5}\frac1{1-(-1)^{8a}q^8}\frac1{1-(-1)^{12a}q^{12}}.
\end{aligned}
\]\end{itemize}
\end{prop}
\begin{proof}
We prove only the simplest case, $\Gamma = \mathbb{Z}_\text{odd}$, and the most complicated case, $\Gamma = \hat{\cD}_\text{even}$. The other cases can be proved analogously.
\begin{itemize}
\item When $(\Gamma , G)=(\bZ_{2k+1},Sp(n))$,\[
\begin{aligned}
\sum_{n=0}^\infty q^{2n} N(\mathbb{Z}_{2k+1},Sp(n)) &= \sum_{n=0}^\infty\sum_{\substack{0\le l_0,\ldots, l_k \\ 2(l_0+\cdots +l_k)=2n}} q^{2n}\\
& = \sum_{0\le l_0,\ldots, l_k} q^{2(l_0+\cdots +l_k)}
 = \frac1{(1-q^2)^{k+1}}.
 \end{aligned}
 \]
 \item When $(\Gamma , G)=(\bZ_{2k+1},SO(2n+1))$,\[
 \begin{aligned}
\sum_{n=0}^\infty q^{2n+1} N(\mathbb{Z}_{2k+1},SO(2n+1)) &= \frac1{2}\sum_{a=0}^1 (-1)^a\sum_{n=0}^\infty (-1)^{na} q^{n} N(\mathbb{Z}_{2k+1},SO(n))\\
&= \frac1{2}\sum_{a=0}^1 (-1)^a\sum_{n=0}^\infty \sum_{\substack{0\le l_0,\ldots, l_k \\ l_0+2(l_1+\cdots +l_k)=n}} (-1)^{na}q^{n}\\
&= \frac1{2}\sum_{a=0}^1 (-1)^a\frac1{1-(-1)^{a}q}\frac1{(1-(-1)^{2a}q^2)^{k}}.
\end{aligned}
\]
\item When $(\Gamma , G)=(\hat{\cD}_{2k},Sp(n))$,\[
\begin{aligned}
\sum_{n=0}^\infty q^{2n} N(\hat D_{2k},Sp(n)) &= \sum_{n=0}^\infty\sum_{\substack{0\le l_1,\ldots, l_4,m_1,\ldots,m_{2k-1} \\ 2(l_1+\cdots +l_4)+2(m_1+\cdots +m_{2k-1})+4(m_2+\cdots +m_{2k-2})=2n}} q^{2n}\\
&= \frac1{(1-q^2)^{4}}\frac1{(1-q^2)^{k}}\frac1{(1-q^4)^{k-1}}.
\end{aligned}
\]
\item When $(\Gamma , G)=(\hat{\cD}_{2k},SO(2n+1))$,\[
\begin{aligned}
&\hspace{10pt} \sum_{n=0}^\infty q^{2n+1} N(\hat D_{2k},SO(2n+1))\\
&= \frac1{2}\sum_{a=0}^1 (-1)^a\sum_{n=0}^\infty (-1)^{na} q^{n} N(\hat D_{2k},SO(n))\\
&= \frac1{2}\sum_{a=0}^1 (-1)^a\sum_{n=0}^\infty \sum_{\substack{0\le l_1,\ldots, l_4,m_1,\ldots,m_{2k-1} \\ l_1+\cdots +l_4+4(m_1+\cdots +m_{2k-1})+2(m_2+\cdots +m_{2k-2})=n \\ (l_2+l_4+m_2+\cdots +m_{2k-2})\,\text{even}, (l_3+l_4)\,\text{even}}} (-1)^{na}q^{n}\\
&= \frac1{2}\sum_{a=0}^1 (-1)^a\sum_{n=0}^\infty \hspace{60pt}\smashoperator{\sum_{\substack{0\le l_1,\ldots, l_4,m_1,\ldots,m_{2k-1} \\ l_1+\cdots +l_4+4(m_1+\cdots +m_{2k-1})+2(m_2+\cdots +m_{2k-2})=n}}}\hspace{30pt} \frac1{2^2}\sum_{b_0,b_1=0}^1 (-1)^{b_0(l_2+l_4+m_2+\cdots +m_{2k-2})+b_1(l_3+l_4)} (-1)^{na}q^{n}\\
&= \frac1{2}\sum_{a=0}^1\frac1{2^2}\sum_{b_0,b_1=0}^1 (-1)^a\frac1{1-(-1)^{a}q}\frac1{1-(-1)^{a+b_0}q}\frac1{1-(-1)^{a+b_1}q} \\
&\hspace{150pt} \times\frac1{1-(-1)^{a+b_0+b_1}q}\frac1{(1-(-1)^{4a}q^4)^k}\frac1{(1-(-1)^{2a+b_0}q^2)^{k-1}}.
\end{aligned}
\]
\end{itemize}
\end{proof}

We now prove the following three propositions, which are a bit more general than the original statements we had to check. The first one corresponds to the case $\Gamma = \mathbb{Z}_n, \hat{\cT}, \hat{\cI}$. And the second one corresponds to the case $\Gamma = \hat{\cD}_\text{odd}, \hat{\cO}$. And the last one corresponds to the case  $\Gamma = \hat{\cD}_\text{even}$.
\begin{prop}
\label{prop:K-formula}
Let $l\in\mathbb{Z}_{\ge 0}$, $k_1,\ldots, k_4, v_1,\ldots, v_{l} \in\mathbb{N}$. Set
\begin{equation*}
\left\{
\begin{aligned}
F_1(q) &= \prod_{r=1}^1\frac1{1-q^{4k_r-2}}\prod_{i=1}^{l}\frac1{1-q^{2v_i}},\\
\tilde F_1(q) &= \prod_{r=1}^1\frac1{1-q^{2k_r-1}}\prod_{i=1}^{l}\frac1{1-q^{2v_i}}.
\end{aligned}
\right.
\end{equation*}
For $k_1\neq k_2$, set
\begin{equation*}
\left\{
\begin{aligned}
F_2(q) &= \frac1{1-q^{2k_2-2k_1}}\prod_{r=1}^{2}\frac1{1-q^{4k_r-2}}\prod_{i=1}^{l}\frac1{1-q^{2v_i}},\\
\tilde F_2(q) &= \frac1{1-q^{4k_2-4k_1}}\prod_{r=1}^{2}\frac1{1-q^{2k_r-1}}\prod_{i=1}^{l}\frac1{1-q^{2v_i}}.
\end{aligned}
\right.
\end{equation*}
For $k_1+k_2 = k_3 + k_4, \{k_1, k_2\} \neq \{k_3,k_4\}$, set
\begin{equation*}
\left\{
\begin{aligned}
F_4(q) &= \frac1{1-q^{2k_1+2k_2-2}}\frac1{1-q^{2k_3-2k_1}}\frac1{1-q^{2k_4-2k_1}}\prod_{r=1}^{4}\frac1{1-q^{4k_r-2}}\prod_{i=1}^{l}\frac1{1-q^{2v_i}},\\
\tilde F_4(q) &= \frac1{1-q^{4k_1+4k_2-4}}\frac1{1-q^{4k_3-4k_1}}\frac1{1-q^{4k_4-4k_1}}\prod_{r=1}^{4}\frac1{1-q^{2k_r-1}}\prod_{i=1}^{l}\frac1{1-q^{2v_i}}.
\end{aligned}
\right.
\end{equation*}
Then we have \[
\begin{aligned}
q^{2k_1-1}\,F_s(q) = \frac1{2}\sum_{a=0}^1 (-1)^a \tilde F_s((-1)^{a}q)
\end{aligned}
\]
for all $s=1,2,4$.
\end{prop}
\begin{proof}
We prove only the case when $s=4$. The others can be proved in a similar way. 
It suffices to consider the case $l=0$.
 Then we have, \[
\begin{aligned}
& \frac1{2}\sum_{a=0}^1 (-1)^a \tilde F_4((-1)^{a}q) \\
&=\frac12\frac1{1-q^{4k_1+4k_2-4}}\frac1{1-q^{4k_3-4k_1}}\frac1{1-q^{4k_4-4k_1}}\left( \prod_{r=1}^{4}\frac1{1-q^{2k_r-1}} - \prod_{r=1}^{4}\frac1{1+q^{2k_r-1}}\right) \\
&=\frac1{1-q^{4k_1+4k_2-4}}\frac1{1-q^{4k_3-4k_1}}\frac1{1-q^{4k_4-4k_1}} \\
& \hspace{130pt} \times \frac{q^{2k_1-1}(1+q^{2k_3-2k_1})(1+q^{2k_4-2k_1})(1+q^{2k_1+2k_2-2})}{(1-q^{4k_1-2})(1-q^{4k_2-2})(1-q^{4k_3-2})(1-q^{4k_4-2})}\\
&=q^{2k_1-1}\frac1{1-q^{2k_1+2k_2-2}}\frac1{1-q^{2k_3-2k_1}}\frac1{1-q^{2k_4-2k_1}}\prod_{r=1}^{4}\frac1{1-q^{4k_r-2}} \\
&=q^{2k_1-1}\,F_4(q),
\end{aligned}
\]
where we used the condition $k_1+k_2=k_3+k_4$ in the second equality.
\end{proof}

\begin{prop}
\label{prop:K-formula-2}
Let $l_0,l_1\in\mathbb{Z}_{\ge 0}$, $k_1, k_2, v_{0,1},\ldots, v_{0,l_0}, v_{1,l_1},\ldots, v_{1,l_1} \in\mathbb{N}$. Set
\begin{equation*}
\left\{
\begin{aligned}
F_2(q) &=  \prod_{r=1}^1\frac1{(1-q^{4k_r-2})^2}\prod_{i=1}^{l_0}\frac1{1-q^{2v_{0,i}}}\prod_{j=1}^{l_1}\frac1{1-q^{2v_{1,j}}},\\
\tilde F_2(q,t) &= \prod_{r=1}^1\frac1{(1-q^{2k_r-1})(1-tq^{2k_r-1})}\prod_{i=1}^{l_0}\frac1{1-q^{2v_{0,i}}}\prod_{j=1}^{l_1}\frac1{1-tq^{2v_{1,j}}}.
\end{aligned}
\right.
\end{equation*}
For $k_1\neq k_2$, set
\begin{equation*}
\left\{
\begin{aligned}
F_4(q) &= \frac1{1-q^{2k_1+2k_2-2}}\frac1{1-q^{2k_2-2k_1}}\prod_{r=1}^{2}\frac1{(1-q^{4k_r-2})^2}\prod_{i=1}^{l_0}\frac1{1-q^{2v_{0,i}}}\prod_{j=1}^{l_1}\frac1{1-q^{2v_{1,j}}},\\
\tilde F_4(q,t) &= \frac1{1-q^{4k_1+4k_2-4}}\frac1{1-q^{4k_2-4k_1}}\prod_{r=1}^{2}\frac1{(1-q^{2k_r-1})(1-tq^{2k_r-1})}\prod_{i=1}^{l_0}\frac1{1-q^{2v_{0,i}}}\prod_{j=1}^{l_1}\frac1{1-tq^{2v_{1,j}}}.
\end{aligned}
\right.
\end{equation*}
Then we have \[
\begin{aligned}
q^{2k_1-1}\,F_s(q) = \frac1{2}\sum_{a=0}^1\frac1{2}\sum_{b=0}^1 (-1)^a \tilde F_s((-1)^{a}q,(-1)^b)
\end{aligned}
\]
for both $s=2$ and $4$.
\end{prop}
\begin{proof}
We prove only the case when $s=4$. The other case can be proved in a similar way. We easily find that \[
\tilde F_4(q,-1) = \tilde F_4(-q,-1)
\]
and these terms in the sum of the right-hand side of the equality cancel each other. Therefore, only the terms with $t=1$ remain, so that we have \[
\frac1{2}\sum_{a=0}^1\frac1{2}\sum_{b=0}^1 (-1)^a \tilde F_4((-1)^{a}q,(-1)^b) 
=\frac1{4}\sum_{a=0}^1 (-1)^a \tilde F_4((-1)^{a}q,1).
\] As in the last proposition, it is enough to prove the case when $l_0=l_1=0$. Under these conditions, we have\[
\begin{aligned}
&\frac1{4}\sum_{a=0}^1 (-1)^a \tilde F_4((-1)^{a}q,1) \\
&=\frac1{4}\frac1{1-q^{4k_1+4k_2-4}}\frac1{1-q^{4k_2-4k_1}}\left(\prod_{r=1}^{2}\frac1{(1-q^{2k_r-1})^2} - \prod_{r=1}^{2}\frac1{(1+q^{2k_r-1})^2}\right) \\
&=\frac1{4}\frac1{1-q^{4k_1+4k_2-4}}\frac1{1-q^{4k_2-4k_1}} \cdot\frac{4q^{2k_1-1}(1+q^{2k_2-2k_1})(1+q^{2k_1+2k_2-2})}{(1-q^{4k_1-2})^2(1-q^{4k_2-2})^2} \\
&=q^{2k_1-1}\frac1{1-q^{2k_1+2k_2-2}}\frac1{1-q^{2k_2-2k_1}}\prod_{r=1}^{2}\frac1{(1-q^{4k_r-2})^2} \\
&=q^{2k_1-1}F_4(q),
\end{aligned}
\]
which is what we wanted to prove.
\end{proof}

\begin{prop}
\label{prop:K-formula-3}
Let $l_{00},l_{01},l_{10},l_{11}\in\mathbb{Z}_{\ge 0}$, $k, v_{00,1},\ldots, v_{00,l_{00}}, v_{01,1} \ldots, v_{11,l_{11}} \in\mathbb{N}$, and set
\begin{equation*}
\left\{
\begin{aligned}
F(q) &= \frac1{(1-q^{4k-2})^5}\prod_{p_0=0}^1\prod_{p_1=0}^1\prod_{i=1}^{l_{p_1p_0}}\frac1{1-q^{2v_{p_1p_0,i}}} ,\\
\tilde F(q,t_0,t_1) &=  \frac1{1-q^{8k-4}}\prod_{p_0=0}^1\prod_{p_1=0}^1 \left(\frac1{1-t_1^{p_1}t_0^{p_0}q^{2k-1}}\prod_{i=1}^{l_{p_1p_0}}\frac1{1-t_1^{p_1}t_0^{p_0}q^{2v_{p_1p_0,i}}}\right). \\
\end{aligned}
\right.
\end{equation*}
Then we have \[
\begin{aligned}
q^{2k-1}\,F(q) = \frac1{2}\sum_{a=0}^1\frac1{2}\sum_{b_0=0}^1\frac1{2}\sum_{b_1=0}^1 (-1)^a \tilde F((-1)^{a}q,(-1)^{b_0},(-1)^{b_1})
\end{aligned}
\]
\end{prop}
\begin{proof}
We easily find that \[
\begin{aligned}
\tilde F(q,1,-1) &= \tilde F(-q,1,-1) ,\\
\tilde F(q,-1,1) &= \tilde F(-q,-1,1) ,\\
\tilde F(q,-1,-1) &= \tilde F(-q,-1,-1).
\end{aligned}
\]
Hence, as before, we have \[
\frac1{2}\sum_{a=0}^1\frac1{2}\sum_{b_0=0}^1\frac1{2}\sum_{b_1=0}^1 (-1)^a \tilde F((-1)^{a}q,(-1)^{b_0},(-1)^{b_1}) 
=\frac1{8}\sum_{a=0}^1 (-1)^a \tilde F((-1)^{a}q,1,1) .
\]
Clearly, it is sufficient to prove the case when $l_{00} = l_{01} = l_{10} = l_{11} = 0$. 
Under these conditions, we have\[
\begin{aligned}
\frac1{8}\sum_{a=0}^1 (-1)^a \tilde F((-1)^{a}q,1,1) 
&=\frac1{8}\frac1{1-q^{8k-4}}\left(\frac1{(1-q^{2k-1})^4} - \frac1{(1+q^{2k-1})^4}\right) \\
&=\frac1{8}\frac1{1-q^{8k-4}}\cdot\frac{4q^{2k-1}\cdot 2(1+q^{4k-2})}{(1-q^{4k-2})^4} \\
&=q^{2k-1}\frac1{(1-q^{4k-2})^5} \\
&=q^{2k-1}\,F(q),
\end{aligned}
\]
which is what we wanted to prove.
\end{proof}

\begin{proof}[Proof of Theorem~\ref{thm:mainA}]
The generating function in each case is given in Proposition~\ref{prop:genX}.
\begin{itemize}
\item 
For $\Gamma = \mathbb{Z}_{m}$, put $(s; k_1; l; v_i) = (1;1;\lfloor\frac{m}{2}\rfloor;1)$ in Proposition~\ref{prop:K-formula}.
\item 
For $\Gamma = \hat{\cT}$, put $(s; k_1, k_2; l; v_1, v_2) = (2;1,2;2;1,2)$ in Proposition~\ref{prop:K-formula}. 
\item 
For $\Gamma = \hat{\cI}$, put $(s; k_1, k_2, k_3, k_4; l; v_1, v_2) = (4;1,3,2,2;2;2,4)$ in Proposition~\ref{prop:K-formula}.
\item 
For $\Gamma = \hat{\cD}_{2m+1}$, put $(s; k_1; l_0, l_1; v_{0,1}, v_{0,2},\ldots,v_{0,m+1}, v_{1,j}) = (2;1;m+1,m;1,2,\ldots,2,1)$ in Proposition~\ref{prop:K-formula-2}.
\item 
For $\Gamma = \hat{\cO}$, put $(s; k_1, k_2; l_0, l_1; v_{0,1}, v_{1,1}) = (4;1,2;1,1;2,1)$ in Proposition~\ref{prop:K-formula-2}.
\item 
Finally for $\Gamma = \hat{\cD}_{2m}$, put $(k; l_{00}, l_{01}, l_{10}, l_{11}; v_{00,i}, v_{01,i}) = (1;m-1,m-1,0,0;2,1)$ in Proposition~\ref{prop:K-formula-3}.
\end{itemize}
This completes the proof.
\end{proof}

\section{The case $(G,\tilde G)=(PSp(n),Spin(2n+1))$, $\Gamma=\hat\cO $}
\label{sec:final}

In the previous section, we have treated the case when $(G,\tilde G)=(Sp(n),SO(2n+1))$. Using that knowledge, let us additionally discuss some cases when $(G,\tilde G)=(PSp(n),Spin(2n+1))$ and the ADE type of $\Gamma$ is exceptional.

Our aim is to establish the theorem below:

\begin{thm}
We have $N(\Gamma,PSp(n))=N(\Gamma,Spin(2n+1))$ for $\Gamma=\hat\cT ,\hat\cO ,\hat\cI $.
\label{thm:Y}
\end{thm}

We first deal with the easy case $\Gamma=\hat\cT $ or $\hat\cI $.
\begin{prop}
\label{prop:boo}
For $\Gamma=\hat\cT $ or $\hat\cI $,
we have $N(\Gamma,Sp(n))=N(\Gamma,PSp(n))$ 
and $N(\Gamma,SO(2n+1))=N(\Gamma,Spin(2n+1))$.
\end{prop}
\begin{proof}
Recall that a homomorphism $f:\Gamma\to G/Z$, where $Z$ is a central Abelian subgroup of $G$,
defines a homomorphism $\check f:\check \Gamma\to G$,
where $\check \Gamma$ is a central extension $0\to Z\to \check \Gamma \to \Gamma \to 0$,
such that $\check f$ maps $Z$ identically.
Let $Z=\bZ_2$. 
Then $H^2(B\Gamma,\bZ_2)=0$ for $\Gamma=\hat\cT $ and $\Gamma=\hat\cI $.
This means that any such $\check f:\check \Gamma\to G$ 
can be modified to a homomorphism $\Gamma\to G$.
Therefore, there is a one-to-one correspondence between homomorphisms from $\Gamma$ to  $G$
and homomorphisms from $\Gamma$ to $G/\bZ_2$.
\end{proof}

We now restrict our attention to the case $\Gamma=\hat\cO $.
In this case, it is no more difficult to prove the finer version of the conjecture,
so we are going to establish the following theorem:
\begin{thm}
\label{thm:bar}
Conjecture~\ref{conj:finer} holds 
for the case $\Gamma=\hat\cO $, 
$G=Sp(n)$, $H=PSp(n)$, $\tilde H=Spin(2n+1)$, $\tilde G=SO(2n+1)$ and $Z=\bZ_2$.
\end{thm}
The proof of this theorem will occupy the bulk of the rest of this paper.

\begin{cor}
\label{cor:bar}
We have $N(\hat\cO,PSp(n))=N(\hat\cO,Spin(2n+1))$.
\end{cor}
\begin{proof}
Apply Corollary~\ref{cor:main} to Theorem~\ref{thm:bar}.
\end{proof}

\begin{proof}[Proof of Theorem~\ref{thm:Y}]
Immediate from Proposition~\ref{prop:boo} and Corollary~\ref{cor:bar}.
\end{proof}

We can now concentrate on proving Theorem~\ref{thm:bar}.
Before doing this, we need some more general discussions,
and some more properties of irreducible representations of $\Gamma=\hat\cO $.

\subsection{Even more properties of irreducible representations of $\Gamma=\hat\cO $}
We let $\sG$ be either $G=Sp(n)$ or $\tilde H=Spin(2n+1)$.
Our presentation of $\hat\cO $ is \[
\hat\cO =\langle a,b \mid a^4=b^3=(ab)^2\rangle.
\]
The set $\psi_{\bZ_2,m}(\hat\cO ,\sG)$ for 
$m\in H^2(B\hat\cO ,\bZ_2)=\bZ_2$
consists of  pairs $(\hat\rho(a),\hat\rho(b))\in \sG^2$
satisfying $(-1)^m\rho(a)^4=\rho(b)^3=(\rho(a)\rho(b))^2$,
up to simultaneous conjugation by $G$.

The only nontrivial one-dimensional representation of $\hat\cO $, which was denoted by $1'$ but we now denote by $x$,
is $x(a)=-1$, $x(b)=1$.
It generates $\bZ_2$, and acts on $\psi_{\bZ_2,m}(\hat\cO ,\sG)$ by \[
x: (\rho(a),\rho(b)) \mapsto (\rho'(a),\rho'(b)) := (-\rho(a),\rho(b)),
\]
Note that $\rho$ and $\rho'$ can still be conjugate in $\sG$.

We let  \[
V_{\bZ_2,m}(\hat\cO ,\sG)=V_{\bZ_2,m}^0(\hat\cO ,\sG)\oplus V_{\bZ_2,m}^1(\hat\cO ,\sG)
\]
where $V_{\bZ_2,m}^e(\hat\cO ,\sG)$ is the subspace on which $x$ acts by $(-1)^e$.
Similarly, we let \[
\psi_{\bZ_2,m}(\hat\cO ,\sG) = \psi_{\bZ_2,m}^\text{fixed}(\hat\cO ,\sG) \sqcup \psi_{\bZ_2,m}^\text{not fixed}(\hat\cO ,\sG)
\] where the superscript `fixed', `not fixed' is with respect to the action by $x$.
\begin{prop}
\label{prop:vec-to-combinatorics}
We have \[
\begin{aligned}
\dim V_{\bZ_2,m}^0(\hat\cO ,\sG)&= \#\psi_{\bZ_2,m}^\text{fixed}(\hat\cO ,\sG) 
+  \frac12\#\psi_{\bZ_2,m}^\text{not fixed}(\hat\cO ,\sG),
\\
\dim V_{\bZ_2,m}^1(\hat\cO ,\sG)&= \phantom{\#\psi_{\bZ_2,m}^\text{fixed}(\hat\cO ,\sG) } 
+  \frac12\#\psi_{\bZ_2,m}^\text{not fixed}(\hat\cO ,\sG).
\end{aligned}
\]
\end{prop}
\begin{proof}
Immediate by expanding the definitions.
\end{proof}

To compute $\dim V^e_{\bZ_2,m}(\hat \cO,\sG)$,
we need to learn to enumerate $\psi_{\bZ_2,m}(\hat\cO ,\sG)$,
and understand the action of $x$.
We look into each of them in turn.

\subsubsection{$\psi_{\bZ_2,m}(\hat\cO ,Sp(n))$}
We already understand $\psi_{\bZ_2,0}(\hat\cO ,Sp(n))=\psi(\hat\cO ,Sp(n))$.
To enumerate $\psi_{\bZ_2,1}(\hat\cO ,Sp(n))$,
we need to enumerate twisted homomorphisms given by pairs satisfying 
$\tilde\rho(a)^4=-\tilde\rho(ab)^{2}$.

To study them, note that $\tilde\rho(a),\tilde\rho(b)$ act on $\bH^n$.
Regard $\bH^n$ as $\bC^{2n}$, 
i.e.~we partially forget the quaternionic structure
and keep only the complex structure.
We then define $\rho(a):=-i \tilde\rho(a)$ and $\rho(b):=\tilde\rho(b)$.
They satisfy $\rho(a)^2=+\rho(ab)^2$,
and $\rho$ is a genuine complex representation of $\Gamma$.
This means that twisted representations $\tilde\rho$'s, in turn,
allow an irreducible decomposition
in terms of $(i\rho(a),\rho(b))$, where $\rho$ runs over ordinary irreducible
representations of $\Gamma$.
This operation preserves the (complex) dimensions 
and the action by $x$, but it changes the reality properties of representations.

Call the resulting irreducible twisted representations as 
\begin{equation}
\begin{array}{r@{}c@{}l}
&\tilde 2''&\\
&|&\\
\tilde 1-\tilde 2-\tilde 3-&\tilde 4&-\tilde 3'-\tilde 2'-\tilde 1'
\end{array}.
\end{equation}
Due to the multiplication by $i$ for $a$, their reality conditions do change.
It is easy to see that $(\tilde 1,\tilde 1')$,
$(\tilde 2,\tilde 2')$,
$(\tilde 3,\tilde 3')$
form complex conjugate pairs.
For $\tilde 2''$, using explicit matrices given in Eq.~\eqref{explicit-rep},
we can directly see that both $\tilde\rho_{2''}(a)$ and $\tilde \rho_{2''}(b)$ 
are in $SU(2)$, and so it is pseudoreal.
As $\tilde 4\otimes 2=\tilde 3+\tilde 3'+\tilde 2''$,
it means that $\tilde 4$ is strictly real. 

The multiplication by $x$ still exchanges $\tilde 1$ and $\tilde 1'$, $\tilde 2$ and $\tilde 2'$, and $\tilde 3$ and $\tilde 3'$, while fixing $\tilde 2''$.

\subsubsection{$\psi_{\bZ_2,m}(\hat\cO ,Spin(2n+1))$}

In this case, what we can easily study are homomorphisms to $SO(2n+1)$ 
rather than those to $Spin(2n+1)$.
So we need to study homomorphisms to $SO(2n+1)$,
and then discuss when and how they lift to (genuine and twisted) homomorphisms to $Spin(2n+1)$.

Consider $\tilde \rho: \Gamma \to O(2n+1)$ up to conjugation,
which can be easily counted via irreducible decomposition.
Say $\tilde\rho$ contains $n_i$ copies of the irreducible representation $\rho_i$, 
so that $2n+1=\sum n_i \dim \rho_i$.
We need to do three things:
\begin{enumerate}
\item Restrict to those which are actually in $SO(2n+1)$.
\item Decide whether it lifts to a genuine $(m=0)$ or a twisted $(m=1)$ homomorphism into $Spin(2n+1)$.
\item In each case, decide whether the lifts are fixed by $x$ or form a pair exchanged by $x$.
\end{enumerate}

The issues (1) and (2) can be solved using a bit of basic algebraic topology. 
Each real representation $\rho$ has an associated total Stiefel-Whitney class $w(\rho)$.
The degree-$k$ term of $w(\rho)$ is denoted by $w_k(\rho)$, the $k$-th Stiefel-Whitney class.
We only need $w_1$ and $w_2$, so we regard
$w(\rho)\in \bZ_2[y]/(y^3)$
where $y$ is the generator of $H^1(B\hat\cO ,\bZ_2)=\bZ_2$,
by a slight abuse of notation.
As $w(\rho_1\oplus\rho_2)=w(\rho_1)w(\rho_2)$, we have \begin{equation}
w(\rho)= \prod_i w(\rho_i)^{n_i}.
\end{equation}
So to compute $w(\rho)$, we only need to know $w(\rho_i)$. 
Then \begin{itemize}
\item A real representation $\rho$ (i.e.~a homomorphism to $O$) lifts to an SO representation if and only if $w_1(\rho)=0$.
\item An SO representation lifts to a genuine $(m=0)$ or a twisted $(m=1)$ homomorphism into $Spin(2n+1)$
depending on whether $w(\rho)=1$ or $w(\rho)=1+y^2$.
Equivalently, $w_2(\rho)=my^2$.
\end{itemize}

Any real representation of $\hat\cO $ is a direct sum of 
copies of the strictly real representations $1$, $3$, $2''$, $3'$, $1'$ 
or the underlying real representations of pseudoreal representations $2$, $4$, and $2'$.

The Stiefel-Whitney classes of the underlying real representations of pseudoreal representations
are all trivial, due to the following reasons.
Given a pseudoreal representation $\Gamma\subset Sp(k)\curvearrowright \bH^k$,
we are interested in the behavior of 
$G\subset SO(4k)\curvearrowright \bR^{4k}$,
where  $G\subset [Sp(k)\times Sp(1)]/\{\pm1\} \subset SO(4k)$.
This inclusion fits into the following diagram:
\begin{equation}
\begin{tikzcd}
    Sp(k)\times Sp(1)\arrow[r,hook]\arrow[d,two heads]& Spin(4k)\arrow[d,two heads]\\ \relax
    (Sp(k)\times Sp(1))/\{\pm1\}\arrow[r,hook]& SO(4k),
\end{tikzcd}
\end{equation}
meaning that $G\subset Sp(k)\subset \bH^k$ regarded as 
$G\subset SO(4k)\subset\bR^{4k}$ automatically lifts to $Spin(4k)$.

The Stiefel-Whitney classes of irreducible strictly real representations of $\hat\cO $ are not hard to determine, either.
By definition, $w(1)=1$ and $w(1')=1+y$. 
Next, $w(3)=1$, because the action of $\hat\cO $ to $3$ 
is by definition obtained by reducing the action of $\hat\cO \subset SU(2)$ to $\hat\cO \subset SO(3)$.
To determine $w(3')$, we use $w(1')w(3')=w(2'\otimes 2)$. But a real representation 
obtained by tensoring a quaternionic representation by the defining representation
is automatically a spin representation, because of the same diagram above;
the only difference is that now $\Gamma$ is embedded diagonally to both $Sp(k)$ and $Sp(1)$.
Therefore $w(2'\otimes 2)=1$,
and therefore $w(3')=(1+y)^{-1}=1+y+y^2$.
Applying the same argument to $w(3)w(3')w(2'')=w(4\otimes 2)$, we get $w(2')=1+y$.

The discussions up to this point take care of the issues (1) and (2) raised above.
We still need to discuss the issue (3).

Suppose we are given a homomorphism $\rho:\hat\cO \to SO(2n+1)$
which lifts to a (genuine or twisted) homomorphism $\tilde \rho$ from $\hat\cO $ to $Spin(2n+1)$.
The action of $x$ fixes $\tilde \rho$ up to conjugation by $Spin(2n+1)$ if and only if
there is a $\tilde g\in Spin(2n+1)$ such that \[
\tilde g \tilde\rho(a) \tilde g^{-1} = -\tilde\rho(a),\qquad
\tilde g\tilde \rho(b)\tilde g^{-1} = \tilde \rho(b).
\]
Such a $\tilde g\in Spin(k)$ determines a corresponding $g\in SO(k)$
such that \begin{equation}
g \rho(a) g^{-1}=\rho(a),\qquad g\rho(b) g^{-1}=\rho(b).
\end{equation}

Now, a  representation of a finite group $\Gamma$
on $\bR^k$ has a  decomposition \begin{equation}
\bR^k \otimes \bC = 
\bigoplus \rho_{\bR,i}^{\oplus s_i}
\oplus
\bigoplus (\rho_{\bH,i}\oplus \rho_{\bH,i})^{\oplus t_i}
\oplus 
\bigoplus (\rho_{\bC,i} \oplus \overline{\rho_{\bC,i}})^{\oplus u_i},
\end{equation}
where $\rho_{\bR,\bH,\bC,i}$ list the irreducible representations of $\Gamma$
of the indicated types.
Any $O(k)$ matrix commuting with the $\Gamma$ action 
is in the subgroup \begin{equation}
X:=\prod_i O(s_i) \times \prod_i Sp(t_i) \times \prod_i U(u_i)
\end{equation}
which acts by permuting the copies of the same irreducible representation;
this is a generalization of Schur's lemma from complex representations
to real representations.
Our $g$ is in $SO(k)$,
so we further need a requirement that 
$g=(g_{\bR,i}; g_{\bH,i};g_{\bC,i})$ satisfies 
$\prod_i (\det g_{\bR,i})^{\dim \rho_{\bR,i}} = +1$.

Now, given a $g\in X \cap SO(k)$, we have \begin{equation}
\tilde g\tilde\rho(a)\tilde g^{-1} = \pm \tilde\rho(a) \label{foo}
\end{equation}
with a $+$ or $-$ sign. 
This sign is a representation $X\cap SO(k)\to \bZ_2$.
As such it is locally a constant,
and therefore it only depends on the connected component of $X\cap SO(k)$,
which is just a product of a number of $\bZ_2$'s. 
To explicitly describe these $\bZ_2$'s,
let $h_i$ be an element from the component of $O(s_i)$ disconnected from the identity.
Then an over-complete basis of these $\bZ_2$'s is given by
(1) $h_i$ for $\dim \rho_{\bR,i}$ is even
and 
(2) $h_j h_k$ for $\dim\rho_{\bR,j}$, $\dim\rho_{\bR,k}$ are both odd.
The sign appearing in \eqref{foo} can be found by a direct computation,
and gives  \begin{equation}
\det\rho_i(a) 
\end{equation} for $g=h_i$ and \begin{equation}
\det\rho_j(a)\det\rho_k(a)
\end{equation} for $g=h_jh_k$.

Applying this consideration for $\Gamma=\hat\cO $ and using our knowledge of types of $\rho_i$ 
and also of $\det\rho_i$,
we obtain the following proposition:
\begin{prop}
\label{prop:x-action-Spin}
Given a homomorphism $\rho:\Gamma\to SO(2n+1)$,
consider its complexification $\rho_\bC:\Gamma\to U(2n+1)$
and let $n_i$ be the number of copies $n_i$
of the irreducible representation $\rho_i$ appearing in the direct sum decomposition of $\rho_\bC$.
Let $\tilde \rho$ be a (genuine or twisted) homomorphism of $\Gamma$ to $Spin(2n+1)$
obtained by lifting $\rho$.
Then $\tilde\rho$ is fixed by the action of $x$ 
if and only if \[
\text{$n_{2''}>0$  or ($n_1+n_3>0$ and $n_{1'}+n_{3'}>0$).}
\]
\end{prop}

\subsection{The proof of the refined conjecture}

With the preparations done, we can finally proceed to the proof of Theorem~\ref{thm:bar}.

\begin{prop}
\label{prop:Sp-refined-counting}
$\psi_{\bZ_2,0}(\hat\cO ,Sp(n))$ can be identified with the sets of integer solutions to \[
2k_1+2n_2 + 6k_3+ 4n_4 + 6k_{3'} + 4k_{2'}+ 2k_{1'}+4k_{2''}=2n.
\] The action of $x$ is given by \[
(k_1,n_2,k_3,n_4,k_{3'},n_{2'},k_{1'};k_{2''})
\mapsto 
(k_{1'},n_{2'},k_{3'},n_4,k_{3},n_{2},k_{1};k_{2''}).
\]

Similarly, 
$\psi_{\bZ_2,1}(\hat\cO ,Sp(n))$ can be identified with the sets of integer solution to \[
2n_{1} + 4 n_{2} + 6 n_{3} + 8k_{4} + 2n_{2''}=2n,
\]
and the action of $x$ is trivial.
\end{prop}

\begin{proof}
For a homomorphism $\rho:\hat\cO \to Sp(n)$,
let $\rho_\bC:\hat\cO \to U(2n)$ be the homomorphism obtained by composing with $Sp(n)\to U(2n)$.
Let $n_i$ be the number of copies of irreducible representation $\rho_i$ in the irreducible decomposition of $\rho_\bC$.
As $\rho$ is pseudoreal, $n_i$ for $i=1,3,1',2''$ are even.
We let $n_i=2k_i$ for these cases.
The action of $x$ can be inferred from the data already given above.

For a twisted homomorphism $\rho$ from $\hat\cO $ to $Sp(n)$,
let $n_i$ be the number of copies of the irreducible twisted representation $\tilde\rho_i$.
As $\rho$ is pseudoreal, $n_1=n_{1'}$, $n_2=n_{2'}$, $n_3=n_{3'}$, and $n_4$ is even.
Writing $n_4=2k_4$, we obtain the integer equation given above.
The action of $x$ can also be inferred from the data already given.
\end{proof}

\begin{prop}
\label{prop:Spin-refined-counting}

$
\psi_{\bZ_2,m}^\text{fixed by $x$}(\hat\cO ,Spin(2n+1))
$  can be identified with the sets of integer solution to \[
n_1+4k_2 + 3n_3+8k_4+3n_{3'}+4k_{2'}+n_{1'}+2n_{2''}=2n+1,
\] with the condition  \[
n_{1'}-n_{3'}+n_{2''} = 2m \mod 4,
\] and
\[
\text{$n_{2''}>0$  or ($n_1+n_3>0$ and $n_{1'}+n_{3'}>0$).}
\]

$
\psi_{\bZ_2,m}^\text{not fixed by $x$}(\hat\cO ,Spin(2n+1))
$ modulo the action of $x$ can be identified with the sets of integer solution to \[
n_1+4k_2 + 3n_3+8k_4+3n_{3'}+4k_{2'}+n_{1'}+2n_{2''}=2n+1,
\] with the condition \[
n_{1'}-n_{3'}+n_{2''} = 2m \mod 4,
\] and  \[
\text{$n_{2''}=0$  and  ($n_1+n_3=0$ or $n_{1'}+n_{3'}=0$).}
\]
\end{prop}

\begin{proof}
For a homomorphism $\rho:\hat\cO\to SO(2n+1)$, let $\rho_\bC:\hat\cO\to U(2n+1)$ 
be its complexification.
Let $n_i$ be the number of copies of $\rho_i$ in $\rho_\bC$.
Then, $n_i$ needs to be even when $\rho_i$ is pseudoreal, and therefore 
$n_2=2k_2$, $n_4=2k_4$, $n_{2'}=2k_{2'}$, from which 
we find $n_1+4k_2 + 3n_3+8k_4+3n_{3'}+4k_{2'}+n_{1'}+2n_{2''}=2n+1$.
From the condition that $\det\rho=0$,
we find that $n_{1'}+n_{3'}+n_{2''}$ is even.
Under this condition, the action of $x$ can be studied using Proposition~\ref{prop:x-action-Spin},
resulting in the statements of this proposition.
\end{proof}

A slight rephrasing of  our main theorem, Theorem~\ref{thm:bar},
is the following:
\begin{prop}
\label{prop:intermediate}
We have \[
\dim V^e_{\bZ_2,m}(\hat\cO ,Sp(n))
=\dim V^m_{\bZ_2,e}(\hat\cO ,Spin(2n+1))
\] for all $e=0,1$ and $m=0,1$.
\end{prop}
\begin{proof}
Using Proposition~\ref{prop:vec-to-combinatorics},
the statement of this proposition is seen to be equivalent to 
the following four statements: \begin{itemize}
\item 
For $(e,m)=(0,0)$, we have 
\begin{multline}
\#\psi^\text{fixed}_{\bZ_2,0}(\hat\cO ,Sp(n))+
\frac12\#\psi^\text{not fixed}_{\bZ_2,0}(\hat\cO ,Sp(n)) \\
=
\#\psi^\text{fixed}_{\bZ_2,0}(\hat\cO ,Spin(2n+1))+
\frac12\#\psi^\text{not fixed}_{\bZ_2,0}(\hat\cO ,Spin(2n+1)),
\label{eqA}
\end{multline}
\item
For $(e,m)=(1,0)$, we have \begin{equation}
\frac12\#\psi^\text{not fixed}_{\bZ_2,0}(\hat\cO ,Sp(n))
=
\#\psi^\text{fixed}_{\bZ_2,1}(\hat\cO ,Spin(2n+1))+
\frac12\#\psi^\text{not fixed}_{\bZ_2,1}(\hat\cO ,Spin(2n+1)),
\label{eqB}
\end{equation}
\item
For $(e,m)=(0,1)$, we have \begin{equation}
\#\psi^\text{fixed}_{\bZ_2,1}(\hat\cO ,Sn(n))+
\frac12\#\psi^\text{not fixed}_{\bZ_2,1}(\hat\cO ,Sp(n))
=
\frac12\#\psi^\text{not fixed}_{\bZ_2,0}(\hat\cO ,Spin(2n+1)),
\label{eqC}
\end{equation}
\item
For $(e,m)=(1,1)$, we have \begin{equation}
\frac12\#\psi^\text{not fixed}_{\bZ_2,1}(\hat\cO ,Sp(n))
=
\frac12\#\psi^\text{not fixed}_{\bZ_2,1}(\hat\cO ,Spin(2n+1)).
\label{eqD}
\end{equation}
\end{itemize}
Now, note that \[
\text{the left hand side of \eqref{eqA}}+
\text{the left hand side of \eqref{eqB}}
= N(\hat\cO ,Sp(n))
\] and that \[
\text{the right hand side of \eqref{eqA}}+
\text{the right hand side of \eqref{eqB}}
=N(\hat\cO ,SO(2n+1)),
\]
and that  the equality $N(\hat\cO ,Sp(n))=N(\hat\cO ,SO(2n+1))$ was already proved in Theorem~\ref{thm:mainA}.
Therefore, we only have to prove either \eqref{eqA} or \eqref{eqB}.

Note also that,
from Proposition~\ref{prop:Sp-refined-counting},
$\frac12\#\psi^\text{not fixed}_{\bZ_2,1}(\hat\cO ,Sp(n))=0$.
Therefore, Eq.~\eqref{eqC} is equivalent to \begin{equation}
\#\psi^\text{fixed}_{\bZ_2,1}(\hat\cO ,Sn(n))
=
\frac12\#\psi^\text{not fixed}_{\bZ_2,0}(\hat\cO ,Spin(2n+1)),
\label{eqX}
\end{equation} and Eq.~\eqref{eqD} is equivalent to \begin{equation}
\#\psi^\text{not fixed}_{\bZ_2,1}(\hat\cO ,Spin(2n+1))=0.
\label{eqY}
\end{equation} 

We can conclude the proof of this proposition and therefore the main theorem Theorem~\ref{thm:bar}, then,
by proving \eqref{eqA}, \eqref{eqX}, and \eqref{eqY},
which are Propositions~\ref{prop:A}, \ref{prop:X}, \ref{prop:Y} given below, respectively.
\end{proof}

To prove Propositions~\ref{prop:A}, \ref{prop:X}, and \ref{prop:Y},
we use generating functions as we did in Sec.~\ref{sec:BC}:

\begin{prop}
\label{prop:genY}
We have the following generating functions for $(e,m)=(0,0),(0,1),(1,1)$, which will be used in the proofs of Proposition$~\ref{prop:A}, \ref{prop:X}, \ref{prop:Y}$, respectively.
\begin{itemize}
\item For $(e,m)=(0,0)$,
\begin{multline*}
\sum_{n=0}^\infty q^{2n} \left(\#\psi^\text{fixed}_{\bZ_2,0}(\hat\cO ,Sp(n))+
\frac12\#\psi^\text{not fixed}_{\bZ_2,0}(\hat\cO ,Sp(n)) \right)\\
=\frac1{2}\left( \frac1{(1-q^2)^4}\frac1{(1-q^4)^2}\frac1{(1-q^6)^2} + \frac1{(1-q^4)^4}\frac1{1-q^{12}} \right),
\end{multline*}
\begin{multline*}
\sum_{n=0}^\infty q^{2n+1} \left(\#\psi^\text{fixed}_{\bZ_2,0}(\hat\cO ,Spin(2n+1))+
\frac12\#\psi^\text{not fixed}_{\bZ_2,0}(\hat\cO ,Spin(2n+1)) \right)\\
=\frac1{2}\sum_{a=0}^1\frac1{4}\sum_{b=0}^3 (-1)^a \frac1{1-(-1)^{a}q}\frac1{1-(-1)^{a}i^{b}q}\frac1{1-(-1)^{2a}i^{b}q^2}\frac1{1-(-1)^{3a}q^3}\frac1{1-(-1)^{3a}i^{-b}q^3} \\
\times\frac1{(1-(-1)^{4a}q^4)^2}\frac1{1-(-1)^{8a}q^8}.
\end{multline*}

\item For $(e,m)=(0,1)$,
\begin{flalign*}
&\hspace{9pt}\sum_{n=0}^\infty q^{2n} \left(\#\psi^\text{fixed}_{\bZ_2,1}(\hat\cO ,Sp(n))\right)
=\frac1{(1-q^2)^2}\frac1{1-q^4}\frac1{1-q^6}\frac1{1-q^8},&
\end{flalign*}
\begin{multline*}
\sum_{n=0}^\infty q^{2n+1} \left(\frac12\#\psi^\text{not fixed}_{\bZ_2,0}(\hat\cO ,Spin(2n+1))\right)\\
=\frac1{2}\sum_{a=0}^1\frac1{4}\sum_{b=0}^3 (-1)^{a}\left(\frac1{1-i^b(-1)^aq}\frac1{1-i^{-b}(-1)^{3a}q^3} + \frac1{1-(-1)^aq}\frac1{1-(-1)^{3a}q^3} - 1\right) \\
\times\frac1{(1-(-1)^{4a}q^4)^2}\frac1{1-(-1)^{8a}q^8}.
\end{multline*}

\item For $(e,m)=(1,1)$,
\begin{multline*}
\sum_{n=0}^\infty q^{2n+1} \left(\frac12\#\psi^\text{not fixed}_{\bZ_2,1}(\hat\cO ,Spin(2n+1))\right)\\
=\frac1{2}\sum_{a=0}^1\frac1{4}\sum_{b=0}^3 (-1)^{a}i^{-2b}\left(\frac1{1-i^b(-1)^aq}\frac1{1-i^{-b}(-1)^{3a}q^3} + \frac1{1-(-1)^aq}\frac1{1-(-1)^{3a}q^3} - 1\right) \\
\times\frac1{(1-(-1)^{4a}q^4)^2}\frac1{1-(-1)^{8a}q^8}.
\end{multline*}
\end{itemize}
\end{prop}
\begin{proof}
The basic idea is the same as in Proposition$~\ref{prop:genX}$.We give only the proof of the case when $(e,m)=(0,0)$. The others can be proved in a similar way.

First, we consider the Sp side. From Proposition~\ref{prop:genX}, we know\[
\sum_{n=0}^\infty q^{2n} \left(\#\psi_{\bZ_2,0}(\hat\cO ,Sp(n))\right) = \frac1{(1-q^2)^4}\frac1{(1-q^4)^2}\frac1{(1-q^6)^2}.
\]
In the same way, we can find\[
\sum_{n=0}^\infty q^{2n} \left(\#\psi^\text{fixed}_{\bZ_2,0}(\hat\cO ,Sp(n))\right) = \frac1{(1-q^4)^4}\frac1{1-q^{12}}.
\]
Putting these together, we have\[
\begin{aligned}
&\hspace{10pt} \sum_{n=0}^\infty q^{2n} \left(\#\psi^\text{fixed}_{\bZ_2,0}(\hat\cO ,Sp(n))+
\frac12\#\psi^\text{not fixed}_{\bZ_2,0}(\hat\cO ,Sp(n)) \right) \\
&= \frac1{(1-q^4)^4}\frac1{1-q^{12}} + \frac12\left(\frac1{(1-q^2)^4}\frac1{(1-q^4)^2}\frac1{(1-q^6)^2} - \frac1{(1-q^4)^4}\frac1{1-q^{12}}\right) \\
&= \frac1{2}\left( \frac1{(1-q^2)^4}\frac1{(1-q^4)^2}\frac1{(1-q^6)^2} + \frac1{(1-q^4)^4}\frac1{1-q^{12}} \right).
\end{aligned}
\]

Next, we prove the Spin side. From Proposition~\ref{prop:Spin-refined-counting},We know that\[
\#\psi^\text{fixed}_{\bZ_2,0}(\hat\cO ,Spin(2n+1))+
\frac12\#\psi^\text{not fixed}_{\bZ_2,0}(\hat\cO ,Spin(2n+1))
\]
equals the number of integer solutions to \[
n_1+4k_2 + 3n_3+8k_4+3n_{3'}+4k_{2'}+n_{1'}+2n_{2''}=2n+1,
\] with the condition  \[
n_{1'}-n_{3'}+n_{2''} = 0 \mod 4.
\]
Therefore, the generating function is calculated as follows.\[
\begin{aligned}
&\hspace{10pt} \sum_{n=0}^\infty q^{2n+1} \left(\#\psi^\text{fixed}_{\bZ_2,0}(\hat\cO ,Spin(2n+1))+
\frac12\#\psi^\text{not fixed}_{\bZ_2,0}(\hat\cO ,Spin(2n+1)) \right)\\
&= \frac1{2}\sum_{a=0}^1 (-1)^a\sum_{n=0}^\infty (-1)^{an} q^{n} \left(\#\psi^\text{fixed}_{\bZ_2,0}(\hat\cO ,Spin(n))+
\frac12\#\psi^\text{not fixed}_{\bZ_2,0}(\hat\cO ,Spin(n)) \right)\\
&= \frac1{2}\sum_{a=0}^1 (-1)^a\sum_{n=0}^\infty \sum_{\substack{0\le l_1,\ldots, l_8 \\ l_1+4l_2+3l_3+8l_4+3l_5+4l_6+l_7+2l_8=n \\ (l_7-l_5+l_8)\equiv 0 \mod 4 }} (-1)^{an}q^{n}\\
&= \frac1{2}\sum_{a=0}^1 (-1)^a\sum_{n=0}^\infty \sum_{\substack{0\le l_1,\ldots, l_8 \\ l_1+4l_2+3l_3+8l_4+3l_5+4l_6+l_7+2l_8=n}}\frac1{4}\sum_{b=0}^3 (-1)^{an}i^{b(l_7-l_5+l_8)}q^{n} \\
&= \frac1{2}\sum_{a=0}^1\frac1{4}\sum_{b=0}^3 (-1)^a \frac1{1-(-1)^{a}q}\frac1{1-(-1)^{a}i^{b}q}\frac1{1-(-1)^{2a}i^{b}q^2}\frac1{1-(-1)^{3a}q^3}\frac1{1-(-1)^{3a}i^{-b}q^3} \\
&\hspace{300pt}\times\frac1{(1-(-1)^{4a}q^4)^2}\frac1{1-(-1)^{8a}q^8},
\end{aligned}
\]
which matches the desired expression.
\end{proof}

Now, let us prove the main propositions. We consider a general case to prove the first one, as in the previous section.
\begin{prop}
\label{prop:K-formula-4}
Let $l\in\mathbb{Z}_{\ge 0}$, $k_1, k_2, v_1,\ldots, v_l \in\mathbb{N}$. For $k_1\neq k_2$, set
\begin{equation*}
\left\{
\begin{aligned}
F(q) &= \frac1{1-q^{2k_1+2k_2-2}}\frac1{(1-q^{2k_2-2k_1})^2}\prod_{r=1}^2\frac1{(1-q^{4k_r-2})^2}\prod_{i=1}^l\frac1{1-q^{2v_i}} ,\\
F_0(q) &= \frac1{1-q^{2k_1+2k_2-2}}\frac1{1-q^{4k_2-4k_1}}\prod_{r=1}^2\frac1{1-q^{8k_r-4}}\prod_{i=1}^l\frac1{1-q^{2v_i}} ,\\
\tilde F(q,t) &= \frac1{1-q^{4k_1+4k_2-4}}\frac1{1-q^{4k_2-4k_1}}\frac1{1-tq^{2k_2-2k_1}}\frac1{(1-q^{2k_1-1})(1-tq^{2k_1-1})} \\
&\hspace{190pt}\times\frac1{(1-q^{2k_2-1})(1-t^{-1}q^{2k_2-1})}\prod_{i=1}^l\frac1{1-q^{2v_i}}.
\end{aligned}
\right.
\end{equation*}
Then we have \[
\begin{aligned}
q^{2k_1-1}\frac{F(q)+F_0(q)}{2} = \frac12\sum_{a=0}^1\frac14\sum_{b=0}^3 (-1)^a\tilde F((-1)^aq,i^b).
\end{aligned}
\]
\end{prop}
\begin{proof}
It suffices to prove the case when $l=0$. Then we have\[
\begin{aligned}
&\hspace{10pt}\frac12\sum_{a=0}^1\frac14\sum_{b=0}^3 (-1)^a\tilde F((-1)^aq,i^b) \\
&=\frac12\frac1{1-q^{4k_1+4k_2-4}}\frac1{1-q^{4k_2-4k_1}}\left\{\frac1{1-q^{2k_2-2k_1}}\frac{q^{2k_1-1}(1+q^{2k_2-2k_1})(1+q^{2k_1+2k_2-2})}{(1-q^{4k_1-2})^2(1-q^{4k_2-2})^2} \right.\\
&\hspace{320pt}\left.+ \frac{q^{2k_1-1}(1+q^{2k_1+2k_2-2})}{(1-q^{8k_1-4})(1-q^{8k_2-4})} \right\} \\
&=\frac{q^{2k_1-1}}2\frac1{1-q^{2k_1+2k_2-2}}\left\{\frac1{(1-q^{2k_2-2k_1})^2}\prod_{r=1}^2\frac1{(1-q^{4k_r-2})^2} + \frac1{1-q^{4k_2-4k_1}}\prod_{r=1}^2\frac1{1-q^{8k_r-4}} \right\} \\
&=q^{2k_1-1}\frac{F(q)+F_0(q)}{2}.
\end{aligned}
\]
This completes the proof.
\end{proof}

\begin{prop}
\label{prop:A}
Eq.~\eqref{eqA} holds, i.e.~\begin{multline*}
q\sum_{n=0}^\infty q^{2n} \left(\#\psi^\text{fixed}_{\bZ_2,0}(\hat\cO ,Sp(n))+
\frac12\#\psi^\text{not fixed}_{\bZ_2,0}(\hat\cO ,Sp(n)) \right)\\
=
\sum_{n=0}^\infty q^{2n+1} \left(\#\psi^\text{fixed}_{\bZ_2,0}(\hat\cO ,Spin(2n+1))+
\frac12\#\psi^\text{not fixed}_{\bZ_2,0}(\hat\cO ,Spin(2n+1)) \right).
\end{multline*}
\end{prop}
\begin{proof}
Both sides are explicitly computed in Proposition~\ref{prop:genY}.
From these expressions, we see that setting $(k_1,k_2;l;v_1)=(1,2;1;2)$ in Proposition~\ref{prop:K-formula-4} yields the desired result.
\end{proof}

\begin{prop}
\label{prop:X}
Eq.~\eqref{eqX} holds, i.e.\[
q\sum_{n=0}^\infty q^{2n} \left(\#\psi^\text{fixed}_{\bZ_2,1}(\hat\cO ,Sp(n))\right)
=
\sum_{n=0}^\infty q^{2n+1} \left(\frac12\#\psi^\text{not fixed}_{\bZ_2,0}(\hat\cO ,Spin(2n+1))\right).
\]
\end{prop}
\begin{proof}
We start from
\[
\begin{aligned}
&\hspace{10pt}\sum_{n=0}^\infty q^{2n+1} \left(\frac12\#\psi^\text{not fixed}_{\bZ_2,0}(\hat\cO ,Spin(2n+1))\right) \\
&=\frac1{2}\sum_{a=0}^1\frac1{4}\sum_{b=0}^3 (-1)^{a}\left(\frac1{1-i^b(-1)^aq}\frac1{1-i^{-b}(-1)^{3a}q^3} + \frac1{1-(-1)^aq}\frac1{1-(-1)^{3a}q^3} - 1\right) \\
&\hspace{290pt}\times\frac1{(1-(-1)^{4a}q^4)^2}\frac1{1-(-1)^{8a}q^8} .
\end{aligned}
\]
As  only the second term in the parentheses survives the sum over $a$ and $b$,
we can continue the computation above as \[
\begin{aligned}
&= \frac1{8}\frac1{(1-q^4)^2(1-q^8)}\sum_{a=0}^1\sum_{b=0}^3 (-1)^a\frac1{1-(-1)^aq}\frac1{1-(-1)^{3a}q^3} \\
&= \frac1{2}\frac1{(1-q^4)^2(1-q^8)}\frac{2q(1+q^2)}{(1-q^2)(1-q^6)}\\
&= q\sum_{n=0}^\infty q^{2n} \left(\#\psi^\text{fixed}_{\bZ_2,1}(\hat\cO ,Sp(n))\right),
\end{aligned}
\]
which is what we wanted to prove.
\end{proof}
\noindent 
Note that the proof of Proposition~\ref{prop:X} essentially came down to Proposition~\ref{prop:K-formula}.
The following proposition can be proved in a similar manner:
\begin{prop}
\label{prop:Y}
Eq.~\eqref{eqY} holds, i.e.\[
\sum_{n=0}^\infty q^{2n+1} \left(\#\psi^\text{not fixed}_{\bZ_2,1}(\hat\cO ,Spin(2n+1))\right)=0.
\]
\end{prop}
\begin{proof}
\[
\begin{aligned}
&\hspace{10pt}\sum_{n=0}^\infty q^{2n+1} \left(\frac12\#\psi^\text{not fixed}_{\bZ_2,0}(\hat\cO ,Spin(2n+1))\right) \\
&=\frac1{2}\sum_{a=0}^1\frac1{4}\sum_{b=0}^3 (-1)^{a}i^{-2b}\left(\frac1{1-i^b(-1)^aq}\frac1{1-i^{-b}(-1)^{3a}q^3} + \frac1{1-(-1)^aq}\frac1{1-(-1)^{3a}q^3} - 1\right) \\
&= 0,
\end{aligned}
\]
as each of the three terms in the parentheses cancels out upon the summation over $a$ and $b$.
\end{proof}
This concludes the proof of the propositions used in the proof of 
Proposition~\ref{prop:intermediate},
which in turn establishes our main claim in this section, Theorem~\ref{thm:bar}.

\def\arxivfont{\rm}
\bibliographystyle{ytamsalpha}
\bibliography{refs}

\providecommand{\bysame}{\leavevmode\hbox to3em{\hrulefill}\thinspace}
\providecommand{\MR}{\relax\ifhmode\unskip\space\fi MR }
\providecommand{\MRhref}[2]{%
  \href{http://www.ams.org/mathscinet-getitem?mr=#1}{#2}
}
\providecommand{\href}[2]{#2}
\providecommand{\doihref}[2]{\href{#1}{#2}}
\providecommand{\arxivfont}{\tt}
\begin{thebibliography}{DFMS97}

\bibitem[Cox40]{CoxeterBinary}
H.~S.~M. Coxeter, \emph{The binary polyhedral groups, and other generalizations of the quaternion group}, \href{http://projecteuclid.org/euclid.dmj/1077492263}{Duke Math. J. \textbf{7} (1940) 367--379}.

\bibitem[CT25]{ChoiTachikawaToAppear}
S.~Choi and Y.~Tachikawa, \emph{{On the vacuum degeneracy of $\mathcal{N}{=}4$ supersymmetric Yang-Mills theory on lens spaces}}. To appear.

\bibitem[DFMS97]{DiFrancesco:1997nk}
P.~Di~Francesco, P.~Mathieu, and D.~S{\'e}n{\'e}chal, \doihref{http://dx.doi.org/10.1007/978-1-4612-2256-9}{\emph{{Conformal Field Theory}}}, Graduate Texts in Contemporary Physics, Springer-Verlag, New York, 1997.

\bibitem[Djo85]{Djokovic}
D.~{\v Z}. Djokovi\'c, \emph{On conjugacy classes of elements of finite order in complex semisimple {L}ie groups}, \doihref{http://dx.doi.org/10.1016/0022-4049(85)90026-X}{J. Pure Appl. Algebra \textbf{35} (1985) 1--13}.

\bibitem[FR18]{Frey:2018vpw}
D.~D. Frey and T.~Rudelius, \emph{{6D SCFTs and the classification of homomorphisms $\Gamma_{ADE} \to E_8$}}, \doihref{http://dx.doi.org/10.4310/ATMP.2020.v24.n3.a4}{Adv. Theor. Math. Phys. \textbf{24} (2020) 709--756}, \href{http://arxiv.org/abs/1811.04921}{{\arxivfont arXiv:1811.04921 [hep-th]}}.

\bibitem[Fre98a]{FreyE6}
D.~D. Frey, \emph{Conjugacy of {${\rm Alt}_5$} and {${\rm SL}(2,5)$} subgroups of {$E_6({\bf C})$}, {$F_4({\bf C})$}, and a subgroup of {$E_8({\bf C})$} of type {$A_2E_6$}}, \href{https://doi.org/10.1006/jabr.1997.7214}{J. Algebra \textbf{202} (1998) 414--454}.

\bibitem[Fre98b]{Frey}
\bysame, \emph{Conjugacy of {${\rm Alt}_5$} and {${\rm SL}(2,5)$} subgroups of {$E_8({\bf C})$}}, \href{https://doi.org/10.1090/memo/0634}{Mem. Amer. Math. Soc. \textbf{133} (1998) viii+162}.

\bibitem[Fre01]{FreyE7}
\bysame, \emph{Conjugacy of {$\rm Alt_5$} and {$\rm SL(2,5)$} subgroups of {$E_7({\Bbb C})$}}, \href{https://doi.org/10.1515/jgth.2001.024}{J. Group Theory \textbf{4} (2001) 277--323}.

\bibitem[Ju23a]{Ju:2023umb}
C.~Ju, \emph{{Chern-Simons Theory, Ehrhart Polynomials, and Representation Theory}}, \doihref{http://dx.doi.org/10.1007/JHEP01(2024)052}{JHEP \textbf{01} (2024) 052}, \href{http://arxiv.org/abs/2304.11830}{{\arxivfont arXiv:2304.11830 [math-ph]}}.

\bibitem[Ju23b]{Ju:2023ssy}
\bysame, \emph{{Exact degeneracy of Casimir energy for $ \mathcal{N} $ = 4 supersymmetric Yang-Mills theory on ADE singularities and S-duality}}, \doihref{http://dx.doi.org/10.1007/JHEP08(2024)037}{JHEP \textbf{08} (2024) 037}, \href{http://arxiv.org/abs/2311.18223}{{\arxivfont arXiv:2311.18223 [hep-th]}}.

\bibitem[Kac94]{Kac}
V.~G. Kac, \emph{{Infinite Dimensional Lie Algebras}}, {Cambridge University Press}, 1994.

\bibitem[McK80]{McKay}
J.~McKay, \emph{Graphs, singularities, and finite groups}, The {S}anta {C}ruz {C}onference on {F}inite {G}roups ({U}niv. {C}alifornia, {S}anta {C}ruz, {C}alif., 1979), Proc. Sympos. Pure Math., vol.~37, Amer. Math. Soc., Providence, RI, 1980, pp.~183--186.

\bibitem[Nak02]{Nakajima}
H.~Nakajima, \emph{Geometric construction of representations of affine algebras}, Proceedings of the {I}nternational {C}ongress of {M}athematicians, {V}ol. {I} ({B}eijing, 2002), Higher Ed. Press, Beijing, 2002, pp.~423--438. \href{http://arxiv.org/abs/math.QA/0212401}{{\arxivfont arXiv:math.QA/0212401}}.

\end{thebibliography}

\end{document}